\documentclass[oneside,english,reqno]{amsart}
\usepackage[T1]{fontenc}
\usepackage[latin9]{inputenc}
\usepackage{geometry}
\usepackage{color}
\usepackage{babel}
\usepackage{prettyref}
\usepackage{booktabs}
\usepackage{units}
\usepackage{textcomp}
\usepackage{mathtools}
\usepackage{enumitem}
\usepackage{amstext}
\usepackage{amsthm}
\usepackage{amssymb}
\usepackage{graphicx}
\usepackage[unicode=true,pdfusetitle,
 bookmarks=true,bookmarksnumbered=false,bookmarksopen=false,
 breaklinks=false,pdfborder={0 0 0},pdfborderstyle={},backref=false,colorlinks=true]
 {hyperref}
\hypersetup{
 linkcolor = blue, citecolor = blue}

\makeatletter

\newcommand{\noun}[1]{\textsc{#1}}
\providecommand{\tabularnewline}{\\}

\numberwithin{equation}{section}
\numberwithin{figure}{section}
\numberwithin{table}{section}
\theoremstyle{plain}
\newtheorem{thm}{\protect\theoremname}[section]
  \theoremstyle{remark}
  \newtheorem*{acknowledgement*}{\protect\acknowledgementname}
  \theoremstyle{definition}
  \newtheorem{defn}[thm]{\protect\definitionname}
  \theoremstyle{plain}
  \newtheorem*{thm*}{\protect\theoremname}
  \theoremstyle{plain}
  \newtheorem{prop}[thm]{\protect\propositionname}
  \theoremstyle{plain}
  \newtheorem{lem}[thm]{\protect\lemmaname}

\usepackage{graphicx}
\usepackage{setspace}
\usepackage[perpage,symbol]{footmisc}
\usepackage{multicol}
\usepackage{etoolbox}
\usepackage{bbm}
\usepackage{enumitem}
\usepackage{caption}
\usepackage[capbesideposition={left,center}]{floatrow}

\DefineFNsymbols*{lamportnostar}[math]{{(\dagger)}{(\ddagger)}{(\S)}{(\P)}{(\|)}{(\dagger\dagger)}{(\ddagger\ddagger)}}
\setfnsymbol{lamportnostar}

\DeclareMathOperator{\stab}{stab}

\DeclareMathOperator{\im}{im}
\DeclareMathOperator{\sgn}{sgn}
\DeclareMathOperator{\Spec}{Spec}
\DeclareMathOperator{\ord}{ord}

\DeclareMathOperator{\orig}{orig}
\DeclareMathOperator{\term}{term}

\newcommand{\one}{\mathbbm{1}}
\newcommand{\Mod}[1]{\,\left(\textup{mod}\;#1\right)}

\newrefformat{prop}{Proposition~\ref{#1}}
\newrefformat{lem}{Lemma~\ref{#1}}
\newrefformat{sec}{§\ref{#1}}
\newrefformat{sub}{§\ref{#1}}
\theoremstyle{definition}
\newtheorem*{rems*}{Remarks}

\setlist[enumerate]{leftmargin=*,widest=0}
\setitemize[0]{leftmargin=10pt,itemindent=0pt}

\allowdisplaybreaks[2]

\renewenvironment{thebibliography}[1]{%
    \begin{oldthebibliography}{#1}%
      \setlength{\parskip}{0ex}%
      \setlength{\itemsep}{0ex}%
      \begin{small}
  }%
  {%
    \end{small}
    \end{oldthebibliography}%
  }

\makeatother

\usepackage{listings}
  \providecommand{\acknowledgementname}{Acknowledgement}
  \providecommand{\definitionname}{Definition}
  \providecommand{\lemmaname}{Lemma}
  \providecommand{\propositionname}{Proposition}
  \providecommand{\theoremname}{Theorem}
\providecommand{\theoremname}{Theorem}

\begin{document}

\title[Ramanujan triangle complexes]{Spectrum and combinatorics of two-dimensional Ramanujan complexes}

\author{Konstantin Golubev and Ori Parzanchevski}
\begin{abstract}
Ramanujan graphs have extremal spectral properties, which imply a
remarkable combinatorial behavior. In this paper we compute the high
dimensional Hodge-Laplace spectrum of Ramanujan triangle complexes,
and show that it implies a combinatorial expansion property, and a
pseudo-randomness result. For this purpose we prove a Cheeger-type
inequality and a mixing lemma of independent interest.
\end{abstract}

\date{\today}

\maketitle

\section{\label{sec:Introduction}Introduction}

Expanders are graphs whose nontrivial adjacency spectrum is concentrated
in a narrow strip. This implies remarkable combinatorial properties,
such as isoperimetric expansion \cite{AM85}, pseudo-randomness \cite{AC88,friedman1987expanding},
rapid convergence of random walk and a large chromatic number \cite{hoffman1970eigenvalues}.
We refer to the surveys \cite{HLW06,Lub12} for the applications of
expanders in mathematics and computer science. 

A $k$-regular graph is called \emph{Ramanujan }if its nontrivial
spectrum is contained within the $L^{2}$-spectrum of its universal
cover, which is the $k$-regular tree $T_{k}$: 
\[
\mathrm{Spec}\left(\mathrm{Adj}_{T_{k}}\right)=\left[-2\sqrt{k-1},2\sqrt{k-1}\right]
\]
(the \emph{trivial} spectrum is, by definition, the eigenvalues $\pm k$).
By the Alon-Boppana theorem (cf.\ \cite[Thm.\ 2.7]{HLW06}), this
is asymptotically the best one can hope for, so that Ramanujan graphs
are optimal expanders. Such graphs were first constructed in \cite{LPS88,margulis1988explicit},
as quotients of the Bruhat-Tits tree associated with $PGL_{2}\left(\mathbb{Q}_{p}\right)$
by arithmetic lattices. It was suggested by several authors \cite{cartwright2003ramanujan,li2004ramanujan,Lubotzky2005a,sarveniazi2007explicit}
that \emph{Ramanujan complexes }should be defined as quotients of
Bruhat-Tits buildings whose spectral properties agree with those of
the building. The cited papers show that such complexes do exist,
and the papers \cite{FGL+11,evra2014mixing} use these spectral properties
to obtain some types of combinatorial expansion.

However, all of the definitions and applications in the cited papers
only refer to the spectrum of operators acting on the vertices of
the complex in question. The spectrum of these operators is encoded
in the spherical representations of the group $PGL_{d},$ and these
correspond to representations of the Hecke algebra of the group, which
is commutative \cite{macdonald1979symmetric}. In this paper, we investigate
the spectrum of the high-dimensional Hodge-Laplace operators, which
encode the homology of the complex in all dimensions \cite{Eck44}.
To achieve this, we interpret (see \prettyref{prop:Iwahori-boundaries})
the simplicial boundary and coboundary maps on Ramanujan complexes
as intertwining maps between different representations of the (non-commutative)
Iwahori-Hecke algebra of $PGL_{d}$. 

For Ramanujan complexes of dimension two, we compute the Hodge-Laplace
spectra in all dimensions, and show that unlike the situation in graphs,
the nontrivial spectrum in dimension one is concentrated in \emph{two}
narrow strips. We give here a loose version, and a tight one appears
in \prettyref{thm:ram-triangle-spec}.
\begin{thm}
\label{thm:Ram-spec}Let $X$ be a Ramanujan complex of type $\widetilde{A}_{2}$,
and $\Delta_{i}=\delta^{*}\delta+\delta\delta^{*}$ the simplicial
Hodge-Laplace operator in dimension $i\in\left\{ 0,1,2\right\} $
(see definitions in \prettyref{sec:Complexes-and-buildings}). Then
the nontrivial spectrum of $\Delta_{i}$ is contained in 
\[
\Delta_{0}:\mathfrak{S}_{0}\qquad\Delta_{1}:\mathfrak{S}_{1}\cup\mathfrak{S}_{0},\qquad\Delta_{2}:\left\{ 0\right\} \cup\mathfrak{S}_{1},
\]
where
\begin{align*}
\mathfrak{S}_{0} & =\left[k_{0}-6\sqrt{k_{0}-1},k_{0}+3\sqrt{k_{0}-1}\right]\\
\mathfrak{S}_{1} & =\left[k_{1}-2\sqrt{k_{1}-1},k_{1}+2\sqrt{k_{1}-1}\right]\bigcup\left[2k_{1}-1,2k_{1}+8\right],
\end{align*}
and $k_{0}$ (resp.\ $k_{1}$) is the vertex degree (resp.\ edge
degree) in $X$.
\end{thm}
In the second part of the paper (Section \ref{sec:Combinatorial-expansion})
we explore the combinatorial information which is encoded in the Hodge-Laplace
spectrum. The results apply to any complex, and not only to quotients
of Bruhat-Tits buildings, so that Section \ref{sec:Combinatorial-expansion}
can be read independently. In §\ref{sec:isop-expansion} we prove
the following theorem, which generalizes the isoperimetric inequalities
from \cite{AM85,parzanchevski2012isoperimetric}:
\begin{thm}
\label{thm:cheeger-gen-dim}Let $X$ be a $d$-dimensional complex
on $n$ vertices, and $Z_{i}=Z_{i}\left(X\right)$ the space of $i$-dimensional
cycles. If $\Spec\Delta_{i}\big|_{Z_{i}}\subseteq\left[k_{i}-\mu_{i},k_{i}+\mu_{i}\right]$
for $0\leq i\leq d-2$ and $\Spec\Delta_{d-1}\big|_{Z_{d-1}}\subseteq\left[\lambda_{d-1},\infty\right)$,
then for any partition $\mathrm{Verts}\left(X\right)=\coprod_{i=0}^{d}A_{i}$
\[
\frac{\left|X\left(A_{0},\ldots,A_{d}\right)\right|n^{d}}{\left|A_{0}\right|\cdot\ldots\cdot\left|A_{d}\right|}\geq k_{0}\cdot\ldots\cdot k_{d-2}\cdot\lambda_{d-1}\left(1-\frac{\mu_{d-2}}{k_{d-2}}-C_{d}\left(\frac{\mu_{0}}{k_{0}}+\ldots+\frac{\mu_{d-2}}{k_{d-2}}\right)\frac{n^{d+1}}{\prod_{i=0}^{d}\left|A_{i}\right|}\right),
\]
where $X\left(A_{0},\ldots,A_{d}\right)$ are the $d$-cells of $X$
in $A_{0}\times\ldots\times A_{d}$, and $C_{d}$ depends only on
$d$.
\end{thm}
For a complex with a complete skeleton, one has $k_{i}=n$ and $\mu_{i}=0$
for $0\leq i\leq d-2$, hence the r.h.s.\ above reads as $n^{d-1}\lambda_{d-1}$,
recovering Theorem 1.2 in \cite{parzanchevski2012isoperimetric}.

In \prettyref{thm:mixing-general} we prove a generalization of the
Expander Mixing Lemma (cf.\ \cite[§2.4]{HLW06}), showing that concentration
of the Hodge-Laplace spectrum implies a pseudo-random behavior. Combining
these combinatorial theorems with \prettyref{thm:ram-triangle-spec},
we obtain the following results on Ramanujan complexes of type $\widetilde{A}_{2}$,
which show that they enjoy isoperimetric expansion, pseudo-randomness
and large chromatic number.
\begin{thm}
\label{thm:intro-cheeger-mixing}Let $X$ be a Ramanujan complex of
type $\widetilde{A}_{2}$ with $n$ vertices, vertex degree $k_{0}=2\left(q^{2}+q+1\right)$
and edge degree $k_{1}=q+1$. Fix a constant $\vartheta>0$.
\begin{enumerate}
\item \noun{\label{enu:(Isoperimetry)-If-}(Isoperimetry)} If $X$ is not
tripartite, then for any partition of the vertices of $X$ into sets
$A_{0},A_{1},A_{2}$ of sizes at least $\vartheta n,$ 
\[
\frac{\left|X\left(A_{0},A_{1},A_{2}\right)\right|n^{2}}{\left|A_{0}\right|\left|A_{1}\right|\left|A_{2}\right|}\geq2q^{3}-4q^{2.5}-C\cdot\frac{q^{2}}{\vartheta^{3}}.
\]
\item \noun{\label{enu:(Pseudo-randomness)-If-}(Pseudo-randomness)} If
$X$ is tripartite, let $A,B,C,D$ be disjoint sets of vertices such
that each of $A\cup D$, $B$ and $C$ is contained in a different
block of the tripartition of $X$. If $A,B,C$ and $D$ are of sizes
at most $\vartheta n$, then
\[
\left|\left|X^{2}\left(A,B,C,D\right)\right|-\frac{27q^{4}\left|A\right|\left|B\right|\left|C\right|\left|D\right|}{n^{3}}\right|\leq\left(65q^{3.5}\vartheta+244q^{2.5}\right)\vartheta n
\]
where $X^{2}\left(A,B,C,D\right)$ are the pairs of triangles $t_{1}\in A\times B\times C$,
$t_{2}\in B\times C\times D$ which share an edge ($\left|t_{1}\cap t_{2}\right|=2$).
\item \label{enu:chrom}If $X$ is not tripartite, the chromatic number
of $X$ is at least $\frac{\sqrt[3]{q}}{5}$.
\end{enumerate}
\end{thm}
Here the chromatic number is the minimal number of colors needed to
color the vertices of the complex with no monochromatic triangle.
It is interesting to compare the last result with that of \cite{evra2014mixing},
which studies mixing in Ramanujan complexes using the spherical representations
alone (which correspond to operators on the vertices of the complex).
They show that the chromatic number of such a complex is at least
$\frac{\sqrt[6]{q}}{2}$, and we expect that in higher dimensions
our new methods should lead to an even greater advantage over the
spherical analysis.
\begin{acknowledgement*}
We thank Alex Lubotzky for providing us guidance and inspiration,
and to Uriya First and Dima Trushin for patiently introducing us to
the representation theory of $p$-adic groups. We are grateful to
Shai Evra, Mark Goresky, Alex Kontorovich and Winnie Li for helpful
discussions and suggestions. The first author was supported by the
ERC and by the Israel Science Foundation. The second author was supported
by The Fund for Math and by NSF grant DMS-1128155, and is grateful
for the hospitality of the Institute for Advanced Study.
\end{acknowledgement*}

\section{\label{sec:Complexes-and-buildings}Complexes and buildings}

We recall the basic elements of simplicial Hodge theory, and of Bruhat-Tits
buildings of type $\widetilde{A}_{d}$. For a more relaxed exposition
of the former we refer to \cite[§2]{parzanchevski2012isoperimetric},
and for the latter to \cite{li2004ramanujan,Lubotzky2005a,Lubotzky2013}. 

\subsection{\label{subsec:Simplicial-Hodge-Theory}Simplicial Hodge Theory}

For a finite simplicial complex $X$ of dimension $d$, denote by
$X^{i}$ the set of cells of dimension $i$ in $X$ ($i$\emph{-cells}).
The \emph{degree} of an $i$-cell is the number of $\left(i+1\right)$-cells
which contain it. Denote by $\Omega^{i}=\Omega^{i}\left(X\right)$
the space of $i$-forms, namely skew-symmetric complex functions on
the oriented $i$-cells, equipped with the inner product $\left\langle f,g\right\rangle =\sum_{\sigma\in X^{i}}f\left(\sigma\right)\overline{g\left(\sigma\right)}$.
The $i$-th boundary map $\partial_{i}:\Omega^{i}\rightarrow\Omega^{i-1}$
is defined by $\left(\partial_{i}f\right)\left(\sigma\right)=\sum_{v:v\sigma\in X^{i}}f\left(v\sigma\right)$,
its dual is the $i$-th coboundary map $\delta_{i}=\partial_{i}^{*}:\Omega^{i-1}\rightarrow\Omega^{i}$,
given by $\left(\delta_{i}f\right)\left(\sigma\right)=\sum_{j=0}^{i}\left(-1\right)^{j}f\left(\sigma\backslash\sigma_{j}\right)$,
and $Z_{i}=\ker\partial_{i}$, $Z^{i}=\ker\delta_{i+1}$, $B_{i}=\im\partial_{i+1}$
and $B^{i}=\im\delta_{i}$ are the cycles, cocycles, boundaries and
coboundaries, respectively.

The \emph{upper, lower} and \emph{full $i$-Laplacians} are $\Delta_{i}^{+}=\partial_{i+1}\delta_{i+1}$,
$\Delta_{i}^{-}=\delta_{i}\partial_{i}$ and $\Delta_{i}=\Delta_{i}^{+}+\Delta_{i}^{-}$,
respectively. Their spectra are closely related: $\Spec\Delta_{i}^{+}$
coincides with $\Spec\Delta_{i+1}^{-}$, up to a difference in the
the multiplicity of zero (which is determined by the number of $i$-cells
and $i+1$-cells), and the spectrum of $\Delta_{i}$ is, up to zeros,
the union of $\Spec\Delta_{i}^{+}$ and $\Spec\Delta_{i}^{-}$. It
is most convenient for our purposes to work with the upper Laplacian,
and $\Delta_{0}^{+}$ is the classical graph Laplacian:
\begin{equation}
\left(\Delta_{0}^{+}f\right)\left(v\right)=\deg\left(v\right)f\left(v\right)-\sum_{w\sim v}f\left(w\right).\label{eq:graph-laplace}
\end{equation}

\subsection{\label{subsec:Bruhat-Tits-buildings}Bruhat-Tits buildings}

Let $F$ be a nonarchimedean local field with ring of integers $\mathcal{O}$,
uniformizer $\pi$, and residue field $\nicefrac{\mathcal{O}}{\pi\mathcal{O}}$
of size $q$, which we identify with $\mathbb{F}_{q}$. We denote
by $\mathcal{B}=\mathcal{B}_{d}\left(F\right)$ the \emph{building
of type $\widetilde{A}_{d-1}$ }associated with $F$, which is defined
as follows. The vertices of $\mathcal{B}$ are the left $K$-cosets
in $G$, where $G=PGL_{d}\left(F\right)$ and $K=PGL_{d}\left(\mathcal{O}\right)$.
Each vertex $gK$ is associated with the homothety class of the $\mathcal{O}$-lattice
$g\mathcal{O}^{d}$. A collection of vertices $\left\{ g_{i}K\right\} _{i=0..r}$
forms an $r$-cell if, possibly after reordering, there exist representatives
$g'_{i}\in GL_{d}\left(F\right)$ for $g_{i}$, such that
\begin{equation}
\pi g'_{0}\mathcal{O}^{d}<g'_{r}\mathcal{O}^{d}<g'_{r-1}\mathcal{O}^{d}<\ldots<g'_{1}\mathcal{O}^{d}<g'_{0}\mathcal{O}^{d}.\label{eq:cell-relation}
\end{equation}
The group $G$ acts on $\mathcal{B}$ by left translation, and if
$\Gamma$ is a torsion-free lattice in $G$ then the quotient $X=\Gamma\backslash\mathcal{B}$
is a finite complex. In the case $d=2$, the building $\mathcal{B}_{2}\left(\mathbb{Q}_{p}\right)$
is a $\left(p+1\right)$-regular tree, and its quotients by lattices
in $G=PGL_{2}\left(\mathbb{Q}_{p}\right)$ are $\left(p+1\right)$-regular
graphs. Certain lattices give rise to Ramanujan quotients:
\begin{thm}[\cite{LPS88,margulis1988explicit}, cf.\ \cite{Sarnak1990,Lub10}]
If $\Gamma$ is a congruence subgroup of a torsion-free arithmetic
lattice in $G$, then $\Gamma\backslash\mathcal{B}_{2}$ is a Ramanujan
graph.
\end{thm}
Namely, its spectrum is contained within $\left\{ -p-1\right\} \cup\left[-2\sqrt{p},2\sqrt{p}\right]\cup\left\{ p+1\right\} $,
where the eigenvalues $\pm\left(p+1\right)$ are the \emph{trivial}
ones: $p+1$ corresponds to the constant function on the vertices,
and if the graph $\Gamma\backslash\mathcal{B}_{2}$ is bipartite,
$-\left(p+1\right)$ appears as an eigenvalue of the function which
takes one value on one side and the opposite value on the other.

\subsection{\label{subsec:Trivial-spectrum-of}Trivial spectrum of complexes}

The function $\tau=\ord_{\pi}\det:PGL_{d}\left(F\right)\rightarrow\nicefrac{\mathbb{Z}}{d\mathbb{Z}}$
induces a $d$-partition on the vertices of $\mathcal{B}_{d}$. We
define the type of a cell $\sigma$ in $\mathcal{B}_{d}$ to be the
subset $\left\{ \tau\left(v\right)\,\middle|\,v\in\sigma\right\} $
of $\nicefrac{\mathbb{Z}}{d\mathbb{Z}}$, and say that a function
on $X^{i}=\Gamma\backslash\mathcal{B}_{d}^{i}$ is trivial if its
lift to $\mathcal{B}_{d}^{i}$ is constant on each type. An eigenvalue
of $\Delta_{i}$ (or of $\Delta_{i}^{\pm}$) is called trivial if
it is obtained from a trivial eigenfunction, and thus the nontrivial
spectrum of these operators is obtained from their restriction to
the functions which sum to zero on each type.

In dimension two, the Bruhat-Tits building $\mathcal{B}_{3}$ has
constant vertex and edge degrees 
\[
k_{0}=2\left(q^{2}+q+1\right)\qquad\mbox{and}\qquad k_{1}=q+1
\]
respectively, and $\tau$ partites the vertices of $\mathcal{B}_{3}$
into three parts. We say that $X=\Gamma\backslash\mathcal{B}_{3}$
is tripartite when $\Gamma$ preserves this partition. In this case
the trivial eigenfunctions on vertices are 
\begin{center}
\begin{tabular}{ccc}
\toprule 
Eigenfunction & Eigenvalue of $\Delta_{0}^{+}$ & Eigenvalue of $\Delta_{0}^{-}$\tabularnewline
\midrule
\midrule 
$\one$ & $0$ & $\left|X^{0}\right|$\tabularnewline
\midrule 
$v\mapsto\omega^{\tau\left(v\right)}$, $\omega=e^{\frac{\pm2\pi i}{3}}$ & $\frac{3k_{0}}{2}$ & $0$\tabularnewline
\bottomrule
\end{tabular}
\par\end{center}

When $X$ is not tripartite, only the constant function is in the
trivial spectrum. The edges of $\mathcal{B}_{3}$ have a canonical
orientation by setting $\tau\left(\mathrm{term}\left(e\right)\right)\equiv\tau\left(\mathrm{orig}\left(e\right)\right)+1\Mod{3}$,
and we can thus define a form in $f\in\Omega^{1}\left(\mathcal{B}_{3}\right)$
by assigning a value to the positively oriented edges. Furthermore,
the action of $\Gamma$ always preserves this orientation, so that
the same holds for $X=\Gamma\backslash\mathcal{B}_{3}$.\footnote{Namely, $X$ is always \emph{disorientable}, see \cite[Def.\ 2.6]{Parzanchevski2012walks}.}
The trivial eigenforms in $\Omega^{1}\left(X\right)$ for a tripartite
$X$ are
\begin{center}
\begin{tabular}{ccc}
\toprule 
Eigenform on positive direction & Eigenvalue of $\Delta_{1}^{+}$ & Eigenvalue of $\Delta_{1}^{-}$\tabularnewline
\midrule
\midrule 
$\one$ & $3k_{1}$ & $0$\tabularnewline
\midrule 
$e\mapsto\omega^{\tau\left(\mathrm{orig}\left(e\right)\right)}$,
$\omega=e^{\frac{\pm2\pi i}{3}}$ & $0$ & $\frac{3k_{0}}{2}$\tabularnewline
\bottomrule
\end{tabular}
\par\end{center}

and again, in the non-tripartite case only the constant one appears.

\subsection{Ramanujan complexes}

There are several plausible ways to define what are Ramanujan complexes,
and these are discussed in \cite{cartwright2003ramanujan,li2004ramanujan,Lubotzky2005a,kang2010zeta,first2016ramanujan}.
However, it can be shown that they all agree for complexes of type
$\widetilde{A}_{2}$, and amount to the following.
\begin{defn}
The complex $X=\Gamma\backslash\mathcal{B}_{3}$ is \emph{Ramanujan}
if the nontrivial spectrum of the Laplace operators in every dimension
is contained within that of the corresponding Laplace operators on
$L^{2}\left(\mathcal{B}_{3}\right)$.

\end{defn}
The papers \cite{li2004ramanujan,Lubotzky2005b,sarveniazi2007explicit}
give several constructions of Ramanujan complexes, some of which are
the clique complexes of Cayley graphs. 

The following theorem determines the spectrum of the upper Laplacians
on two-dimensional Ramanujan complexes. The spectra of the full and
lower Laplacians can be inferred from it (see §\ref{subsec:Simplicial-Hodge-Theory}).
By definition, this is the same as determining the $L^{2}$-spectrum
of the Laplacians on $\mathcal{B}_{3}$ itself, but in addition we
determine the multiplicity of eigenvalues on the finite quotients.
\begin{thm}
\label{thm:ram-triangle-spec}Let $X$ be a Ramanujan quotient of
$\mathcal{B}_{3}$ with $n$ vertices, and vertex and edge degrees
$k_{0}=2\left(q^{2}+q+1\right)$ and $k_{1}=q+1$. If $X$ is non-tripartite,
then
\begin{enumerate}
\item $\Delta_{0}^{+}$ has the trivial eigenvalue $0$, and $n-1$ nontrivial
eigenvalues in $\left[k_{0}-6q,k_{0}+3q\right]$.
\item $\Delta_{1}^{+}$ has
\begin{enumerate}
\item The trivial eigenvalue $3k_{1}$.
\item $n-1$ zeros, corresponding to $B^{1}\left(X\right)$ (coboundaries).
\item For every nontrivial $\lambda\in\Spec\Delta_{0}^{+}$, the eigenvalues
$\frac{3k_{1}}{2}\pm\sqrt{\left(\frac{3k_{1}}{2}\right)^{2}-\lambda}$.
\\
This amounts to $n-1$ eigenvalues in each of the strips
\begin{equation}
\begin{aligned}\mathcal{I}_{-} & =\left[\tfrac{3k_{1}}{2}-\sqrt{\left(\tfrac{k_{1}}{2}\right)^{2}+8q},k_{1}+1\right]\\
\mathcal{I}_{+} & =\left[2k_{1}-1,\tfrac{3k_{1}}{2}+\sqrt{\left(\tfrac{k_{1}}{2}\right)^{2}+8q}\right].
\end{aligned}
\label{eq:good-bad-interval}
\end{equation}
\item $n\left(q^{2}+q-2\right)+2$ eigenvalues in the strip 
\begin{equation}
\mathcal{I}=\left[k_{1}-2\sqrt{q},k_{1}+2\sqrt{q}\right].\label{eq:excellent-interval}
\end{equation}
\end{enumerate}
\end{enumerate}
If $X$ is tripartite, then
\begin{enumerate}
\item $\Delta_{0}^{+}$ has trivial spectrum $\left\{ 0,\frac{3k_{0}}{2},\frac{3k_{0}}{2}\right\} $
(see \ref{subsec:Trivial-spectrum-of}), and $n-3$ eigenvalues in
$\left[k_{0}-6q,k_{0}+3q\right]$.
\item $\Delta_{1}^{+}$ has
\begin{enumerate}
\item The trivial eigenvalue $3k_{1}$, and two trivial zeros (both coming
from $B^{1}\left(X\right)$).
\item $n-3$ nontrivial zeros, all coming from $B^{1}\left(X\right)$.
\item $n-3$ eigenvalues in each of $\mathcal{I}_{\pm}$, corresponding
to $\frac{3k_{1}}{2}\pm\sqrt{\left(\frac{3k_{1}}{2}\right)^{2}-\lambda}$
for $\lambda$ a nontrivial eigenvalue of $\Delta_{0}^{+}$.
\item $n\left(q^{2}+q-2\right)+6$ eigenvalues in $\mathcal{I}$.
\end{enumerate}
\end{enumerate}
\end{thm}
Let us make a few remarks:
\begin{enumerate}
\item The spectrum of $\Delta_{0}^{+}$ is well-known \cite{macdonald1979symmetric,li2004ramanujan,Lubotzky2005a},
but our methods are different, and give the spectrum in all dimensions
in a unified manner.
\item These bounds are sharp: a sequence of Ramanujan complexes with injectivity
radius growing to infinity (as constructed in \cite{lubotzky2007moore})
has Laplace spectra which accumulate to any point in these intervals.
This follows from \cite{li2004ramanujan} for dimension zero and from
\cite[§3.5]{Parzanchevski2012walks} for general dimension.
\item All the zeros in the spectra of $\Delta_{0}^{+}$ and $\Delta_{1}^{+}$
come from $B^{1}\left(X\right)$, so that the zeroth and first Betti
numbers of $X$ vanish, in accordance with \cite{Gar73,Casselman1974}.
\end{enumerate}
It is interesting to compare Theorem \ref{thm:ram-triangle-spec}
with Garland's spectral bounds:
\begin{thm*}[Garland's bound, \cite{Gar73,papikian2008eigenvalues,Gundert2013}]
If $X$ is a finite complex such that $\Spec\Delta_{0}^{+}\left(\mathrm{link}\left(\sigma\right)\right)\big|_{Z_{0}\left(\mathrm{link}\left(\sigma\right)\right)}\subseteq\left[\lambda,\Lambda\right]$
for every $\sigma\in X^{j-2}$ and $k\leq\deg\sigma\leq K$ for every
$\sigma\in X^{j-1}$, then
\[
\Spec\Delta_{j}^{+}\left(X\right)\big|_{Z_{j}\left(X\right)}\subseteq\left[\left(j+1\right)\lambda-jK,\left(j+1\right)\Lambda-jk\right].
\]
\end{thm*}
The links of vertices in $\mathcal{B}_{3}$ are incidence graphs of
projective planes over $\mathbb{F}_{q}$, which have $\lambda=k_{1}-\sqrt{q}$
and $\Lambda=2k_{1}$, so that Garland's bound implies that any quotient
of $\mathcal{B}_{3}$ satisfies 
\[
\Spec\Delta_{1}^{+}\big|_{Z_{1}}\subseteq\left[k_{1}-2\sqrt{q},3k_{1}\right].
\]
\prettyref{thm:ram-triangle-spec} shows that both ends are tight!
The drawback of Garland's method is that it misses the sparse picture
within this interval, which is crucial for our combinatorial purposes,
namely, the results in \prettyref{sec:isop-expansion} and \prettyref{sec:pseudo-randomness}.
The proof of \prettyref{thm:ram-triangle-spec} occupies the next
three sections.

\section{Computation of the Laplace spectrum}

\subsection{\label{sec:Iwahori-Hecke-boundary-maps}Boundary maps as Iwahori-Hecke
operators}

In this section we translate the simplicial boundary and coboundary
maps into intertwining operators between different representations
arising from the group $PGL_{3}$. Keeping the notations of §\ref{subsec:Bruhat-Tits-buildings},
we fix the ``fundamental'' vertex $v_{0}=K$ in $\mathcal{B}^{0}=\nicefrac{G}{K}$.
It follows from the fact that $\Gamma$ is torsion-free that it acts
freely on vertices, and thus if we normalize the Haar measure on $G$
so that $\mu\left(K\right)=1$, we have $\mu\left(\Gamma\backslash G\right)=n$.
Furthermore, this implies a linear isometry $\Omega^{0}\left(X\right)\cong L^{2}\left(\Gamma\backslash G/K\right)$,
given explicitly by $f\left(gv_{0}\right)=f\left(\Gamma gK\right)$.
We identify $L^{2}\left(\Gamma\backslash G/K\right)$ with $L^{2}\left(\Gamma\backslash G\right)^{K}$,
the space of $K$-fixed vectors in the $G$-representation $L^{2}\left(\Gamma\backslash G\right)$.

The element $\sigma=\left(\begin{smallmatrix}0 & 1 & 0\\
0 & 0 & 1\\
\pi & 0 & 0
\end{smallmatrix}\right)\in G$ acts on $\mathcal{B}$ by rotation on the triangle consisting of
the vertices $v_{0}$, $\sigma v_{0}$, and $\sigma^{2}v_{0}$. We
fix the oriented edge $e_{0}=\left[v_{0},\sigma v_{0}\right]$, and
define 
\begin{equation}
E=\stab_{G}e_{0}=K\cap\sigma K\sigma^{-1}=\left\{ \left(\begin{smallmatrix}* & * & *\\
* & * & *\\
x & y & *
\end{smallmatrix}\right)\in K\,\middle|\,x,y\in\pi\mathcal{O}\right\} .\label{eq:E-def}
\end{equation}
Since $G$ acts transitively on the non-oriented edges of $\mathcal{B}$,
and preserves the canonical orientation from §\ref{subsec:Trivial-spectrum-of},
the positively oriented edges of $X$ correspond to double cosets
$\Gamma\backslash G/E$, giving an identification of $\Omega^{1}\left(X\right)$
with $L^{2}\left(\Gamma\backslash G\right)^{E}$, by 
\begin{equation}
f\left(ge_{0}\right)=f\left(\left[gv_{0},g\sigma v_{0}\right]\right)=\sqrt{\mu\left(E\right)}f\left(\Gamma g\right),\quad f\left(\left[g\sigma v_{0},gv_{0}\right]\right)=-\sqrt{\mu\left(E\right)}f\left(\Gamma g\right),\label{eq:1-forms-as-E-stable-vecs}
\end{equation}
where $\mu\left(E\right)=\frac{\mu\left(K\right)}{\left[K:E\right]}=\frac{1}{q^{2}+q+1}.$
The scaling by $\sqrt{\mu\left(E\right)}$ is needed to make the isomorphism
$\Omega^{1}\left(X\right)\cong L^{2}\left(\Gamma\backslash G\right)^{E}$
an isometry: if $\left\{ g_{i}e_{0}\right\} _{i=1}^{nk_{0}/2}$ represent
the edges of $X$ and $f\in\Omega^{1}\left(X\right)$ then
\[
\left\Vert f\right\Vert _{\Omega^{1}\left(X\right)}^{2}=\sum_{i}\left|f\left(g_{i}e_{0}\right)\right|^{2}=\sum_{i}\mu\left(E\right)\left|f\left(\Gamma g_{i}\right)\right|^{2}=\int_{\Gamma\backslash G}\left|f\left(\Gamma g\right)\right|^{2}dg=\left\Vert f\right\Vert _{L^{2}\left(\Gamma\backslash G\right)}^{2}.
\]
We fix the triangle $t_{0}=\left[v_{0},\sigma v_{0},\sigma^{2}v_{0}\right]$,
whose pointwise stabilizer is the Iwahori subgroup 
\[
I=K\cap\sigma K\sigma^{-1}\cap\sigma^{2}K\sigma^{-2}=\left\{ \left(\begin{smallmatrix}* & * & *\\
x & * & *\\
y & z & *
\end{smallmatrix}\right)\in K\,\middle|\,x,y,z\in\pi\mathcal{O}\right\} .
\]
As for edges, $G$ acts transitively on non-oriented triangles, and
preserves triangle orientation. Thus, the stabilizer of $t_{0}$ as
a cell is 
\[
T:=\stab_{G}t_{0}=\left\langle \sigma\right\rangle I=I\sqcup\sigma I\sqcup\sigma^{2}I,
\]
and in particular $\left\langle \sigma\right\rangle $ and $I$ commute.
Again, $f\left(gt_{0}\right)=\sqrt{\mu\left(T\right)}f\left(\Gamma g\right)$
gives a linear isometry $\Omega^{2}\left(X\right)\cong L^{2}\left(\Gamma\backslash G\right)^{T}$,
where $\mu\left(T\right)=3\mu\left(I\right)=\frac{3\mu\left(K\right)}{\left[K:I\right]}=\frac{3}{\left(q^{2}+q+1\right)\left(q+1\right)}$.

\medskip{}

Denoting $K_{0}=K$, $K_{1}=E$, and $K_{2}=T$, we have $\Omega^{i}\left(X\right)\cong L^{2}\left(\Gamma\backslash G\right)^{K_{i}}$.\footnote{This nice picture only holds for $PGL_{3}$. In $PGL_{d}$ with $d\geq4$
the group does not act transitively on each dimension, and there are
also elements which flip orientations in the middle dimension.} As $I\leq E\leq K$ and $I\leq T$, the three spaces $L^{2}\left(\Gamma\backslash G\right)^{K_{i}}$
are contained in $L^{2}\left(\Gamma\backslash G\right)^{I}$. The
Iwahori-Hecke algebra $\mathcal{H}=C_{c}\left(I\backslash G/I\right)$
consists of the compactly supported, bi-$I$-invariant complex functions
on $G$, with multiplication defined by convolution. If $\left(\rho,V\right)$
is a representation of $G$, then $\left(\overline{\rho},V^{I}\right)$
is a representation of $\mathcal{H}$, where $\overline{\rho}\left(\eta\right)v:=\int_{G}\eta\left(g\right)\rho\left(g\right)v\,dg$. 

We proceed to show that the (co-)boundary maps between $L^{2}\left(\Gamma\backslash G\right)^{K_{i}}$
are given by certain intertwining elements in $\mathcal{H}$.
\begin{prop}
\label{prop:Iwahori-boundaries}The following elements of $\mathcal{H}$:
\begin{equation}
\begin{aligned}\partial_{1} & =\tfrac{1}{\sqrt{\mu\left(E\right)}}\left(\one_{K\sigma^{2}}-\one_{K}\right) &  &  & \delta_{1} & =\tfrac{1}{\sqrt{\mu\left(E\right)}}\left(\one_{\sigma K}-\one_{K}\right)\\
\partial_{2} & =\tfrac{1}{\sqrt{\mu\left(E\right)\mu\left(T\right)}}\cdot\one_{ET} &  &  & \delta_{2} & =\tfrac{1}{\sqrt{\mu\left(E\right)\mu\left(T\right)}}\cdot\one_{TE}
\end{aligned}
\label{eq:co_boundary_defs}
\end{equation}
act as the corresponding simplicial operators. Namely, each $\partial_{i}\in\mathcal{H}$
takes $L^{2}\left(\Gamma\backslash G\right)^{K_{i}}$ to $L^{2}\left(\Gamma\backslash G\right)^{K_{i-1}}$
and acts as the boundary operator $\partial_{i}:\Omega^{i}\left(X\right)\rightarrow\Omega^{i-1}\left(X\right)$
with respect to the identifications of $\Omega^{i}\left(X\right)$
with $L^{2}\left(\Gamma\backslash G\right)^{K_{i}}$, and likewise
for $\delta_{i}\in\mathcal{H}$ and $\delta_{i}:\Omega^{i-1}\rightarrow\Omega^{i}$.
\end{prop}
\begin{proof}
Both $\partial_{i+1}$ and $\delta_{i}$ map any representation $V$
into $V^{K_{i}}$, since they are constant on right $K_{i}$ cosets
(note that $\sigma K=E\sigma K$). Let $f\in L^{2}\left(\Gamma\backslash G\right)^{E}\cong\Omega^{1}\left(E\right)$,
and let $K=\coprod_{kE\in\nicefrac{K}{E}}kE$. For any $gv_{0}\in X^{0}$
we have
\begin{align*}
\left(\one_{K}f\right)\left(gv_{0}\right) & =\left(\one_{K}f\right)\left(\Gamma g\right)=\int_{G}\one_{K}\left(x\right)\left(xf\right)\left(\Gamma g\right)dx=\int_{K}\left(xf\right)\left(\Gamma g\right)dx\\
 & =\int_{K}f\left(\Gamma gx\right)dx=\sum_{kE\in\nicefrac{K}{E}}\int_{E}f\left(\Gamma gke\right)de=\sum_{kE\in\nicefrac{K}{E}}\int_{E}f\left(\Gamma gk\right)de\\
 & =\mu\left(E\right)\sum_{kE\in\nicefrac{K}{E}}f\left(\Gamma gk\right)=\sqrt{\mu\left(E\right)}\sum_{kE\in\nicefrac{K}{E}}f\left(gke_{0}\right).
\end{align*}
The group $K$ acts transitively on the $q^{2}+q+1$ positive edges
leaving $v_{0}$, so that the positive edge leaving $gv_{0}$ (for
any $g\in G$) are $\left\{ gke_{0}\right\} _{kE\in\nicefrac{K}{E}}$.
Therefore, 
\begin{equation}
\frac{1}{\sqrt{\mu\left(E\right)}}\left(\one_{K}f\right)\left(gv_{0}\right)=\sum_{kE\in\nicefrac{K}{E}}f\left(gke_{0}\right)=\sum_{{\orig e=gv_{0}\atop e\text{ positive}}}f\left(e\right)=-\sum_{{\term e=gv_{0}\atop e\text{ negative}}}f\left(e\right).\label{eq:boundary_1_1}
\end{equation}
In a similar manner, the positive edges with terminus $gv_{0}$ are
$\left\{ gk\sigma^{2}e_{0}\right\} _{k\in K}$, and if $K\sigma^{2}E=\coprod_{k\sigma^{2}E\in\nicefrac{K\sigma^{2}E}{E}}k\sigma^{2}E$
then
\begin{equation}
\begin{aligned}\left(\one_{K\sigma^{2}}f\right)\left(gv_{0}\right) & =\int_{K\sigma^{2}E}f\left(\Gamma gx\right)dx=\sum_{k\sigma^{2}E\in\nicefrac{K\sigma^{2}E}{E}}\int_{E}f\left(\Gamma gk\sigma^{2}e\right)de\\
 & =\sqrt{\mu\left(E\right)}\sum_{k\sigma^{2}E\in\nicefrac{K\sigma^{2}E}{E}}f\left(gk\sigma^{2}e_{0}\right)=\sqrt{\mu\left(E\right)}\sum_{{\term e=gv_{0}\atop e\text{ positive}}}f\left(e\right).
\end{aligned}
\label{eq:boundary_1_2}
\end{equation}
Together with \prettyref{eq:boundary_1_1}, this implies that $\partial_{1}$
from \eqref{eq:co_boundary_defs} indeed act as the simplicial $\partial_{1}$,
justifying the abuse of notation. The reasoning for $\partial_{2}$
is similar, save for the fact that $T\nleq E$ (in fact, $E\cap T=I$).
However, $E$ acts transitively on the triangles containing $e_{0}$,
hence for $f\in L^{2}\left(\Gamma\backslash G\right)^{T}$ 
\begin{align*}
\tfrac{1}{\sqrt{\mu\left(E\right)\mu\left(T\right)}}\left(\one_{ET}f\right)\left(ge_{0}\right) & =\tfrac{1}{\sqrt{\mu\left(T\right)}}\left(\one_{ET}f\right)\left(\Gamma g\right)=\sqrt{\mu\left(T\right)}\sum_{eT\in\nicefrac{ET}{T}}f\left(\Gamma ge\right)\\
 & =\sum_{eT\in\nicefrac{ET}{T}}f\left(get_{0}\right)=\negthickspace\negthickspace\sum_{\tau\in X^{2}\,:\,e_{0}\in\partial\tau}\negthickspace\negthickspace\negthickspace\negthickspace f\left(g\tau\right)=\negthickspace\negthickspace\sum_{\tau\in X^{2}\,:\,ge_{0}\in\partial\tau}\negthickspace\negthickspace\negthickspace\negthickspace f\left(\tau\right),
\end{align*}
agreeing with $\partial_{2}:\Omega^{2}\rightarrow\Omega^{1}$. The
coboundary operators can be analyzed in a similar manner, or as follows:
$\mathcal{H}$ is a $*$-algebra by $\eta^{*}\left(g\right)=\overline{\eta\left(g^{-1}\right)}$,
and a unitary representation $\rho$ of $G$ induces a unitary $\mathcal{H}$-representation,
i.e.\ $\overline{\rho}\left(\eta\right)^{*}=\overline{\rho}\left(\eta^{*}\right)$
(this uses unimodularity of $G$). For $V=L^{2}\left(\Gamma\backslash G\right)$
this gives 
\[
\partial_{1}^{*}=\tfrac{1}{\sqrt{\mu\left(E\right)}}\left(\overline{\one_{\left(K\sigma^{2}\right)^{-1}}}-\overline{\one_{K^{-1}}}\right)=\tfrac{1}{\sqrt{\mu\left(E\right)}}\left(\one_{\sigma K}-\one_{K}\right)
\]
and similarly for $\partial_{2}^{*}$.
\end{proof}
Since $\Gamma$ is cocompact, $L^{2}\left(\Gamma\backslash G\right)$
decomposes as a sum of irreducible unitary representations, $L^{2}\left(\Gamma\backslash G\right)=\bigoplus_{\alpha}W_{\alpha}$,
and $\Omega^{i}\left(X\right)\cong L^{2}\left(\Gamma\backslash G\right)^{K_{i}}=\bigoplus_{\alpha}W_{\alpha}^{K_{i}}\leq\bigoplus_{\alpha}W_{\alpha}^{I}$.
Each $W_{\alpha}^{I}$ is a sub-$\mathcal{H}$-representation, so
that the operators $\partial_{i},\delta_{i}$ decompose with respect
to this sum, and thus the Laplacians as well, giving $\mathrm{\Spec}\Delta_{i}^{\pm}=\bigcup_{\alpha}\Spec\Delta_{i}^{\pm}\big|_{W_{\alpha}^{K_{i}}}$,
with the correct multiplicities. To understand the spectra it is enough
look at the $W_{\alpha}$ which are Iwahori-spherical, namely, contain
$I$-fixed vectors. Furthermore, the isomorphism type of $W_{\alpha}$
already determines the spectrum of $\Delta_{i}^{\pm}$ on $W_{\alpha}^{K_{i}}$.
By \cite[Prop.\ 2.6]{casselman1980unramified}, if $W_{\alpha}^{I}\neq0$
then $W_{\alpha}$ is embeddable in a \emph{principal series representation}
$V_{\mathfrak{z}}$. Namely, there exists $\mathfrak{z}=\left(z_{1},z_{2},z_{3}\right)\in\mathbb{C}^{3}$
(the \emph{Satake parameters}) with $z_{1}z_{2}z_{3}=1$, and
\begin{equation}
V_{\mathfrak{z}}=\mathrm{uInd}_{B}^{G}\chi_{\mathfrak{z}}=\left\{ f:G\rightarrow\mathbb{C}\,\middle|\,f\left(bg\right)=\delta^{-\frac{1}{2}}\left(b\right)\chi_{\mathfrak{z}}\left(b\right)f\left(g\right)\;\forall b\in B\right\} ,\label{eq:UInd}
\end{equation}
where $\chi_{\mathfrak{z}}$ is the character $\chi_{\mathfrak{z}}\left(b\right)=\prod_{i=1}^{3}z_{i}^{\ord_{\pi}b_{ii}}$
of the Borel group $B:=\left\{ \left(\begin{smallmatrix}* & * & *\\
0 & * & *\\
0 & 0 & *
\end{smallmatrix}\right)\in G\right\} $, and $\delta\left(b\right)=\left|b_{33}\right|^{2}/\left|b_{11}\right|^{2}$
is the modular character of $B$. For obvious reasons, it is convenient
to introduce the notation
\[
\widetilde{\chi}_{\mathfrak{z}}\left(b\right)=\delta^{-\frac{1}{2}}\left(b\right)\chi_{\mathfrak{z}}\left(b\right)=\frac{\left|b_{11}\right|}{\left|b_{33}\right|}\prod_{i=1}^{3}z_{i}^{\ord_{\pi}b_{ii}}=\left(\frac{z_{1}}{q}\right)^{\ord_{\pi}b_{11}}z_{2}^{\ord_{\pi}b_{22}}\left(qz_{3}\right)^{\ord_{\pi}b_{33}}.
\]

Having decomposed $L^{2}\left(\Gamma\backslash G\right)=\bigoplus_{\alpha}W_{\alpha}$,
and found a $\Delta_{i}^{\pm}$-eigenform $f\in W_{\alpha}^{K_{i}}\leq\Omega^{i}\left(X\right)$,
we can lift it to a $\Gamma$-periodic eigenform $\tilde{f}\in{}^{\Gamma}\Omega^{i}\left(\mathcal{B}\right)$.
For some $\mathfrak{z}$ we have $\Psi:W_{\alpha}\hookrightarrow V_{\mathfrak{z}}$,
and naturally $\Psi f\in V_{\mathfrak{z}}^{K_{i}}$; since $V_{\mathfrak{z}}$
is defined as a set of complex functions on $G$, we can think of
$\Psi f$ as an $i$-form on $\mathcal{B}$. Thus, both $\tilde{f}$
and $\Psi f$ are $\Delta_{i}^{\pm}$-eigenforms with the same eigenvalue.
However, they are not the same, as $\tilde{f}$ attains finitely many
values and $\Psi f$ infinitely many, in general. Nevertheless, the
matrix coefficients $g\mapsto\left\langle g\widetilde{f},\widetilde{f}\right\rangle $
and $g\mapsto\left\langle g\Psi f,\Psi f\right\rangle $ are the same,
since $\Psi$ is a unitary embedding. When these matrix coefficient
are in $L^{2+\varepsilon}\left(G\right)$ for every $\varepsilon>0$,
and only then, the representation $W_{\alpha}$ is weakly contained
in $L^{2}\left(G\right)$, which implies that the corresponding eigenvalue
is in the $L^{2}$-spectrum of $\Delta_{i}^{\pm}$ on $\mathcal{B}$
(cf.\ \cite{Haagerup1988}). 

\subsection{Analysis of the principal series}

Even though $W_{\alpha}$ is only a subrepresentation of $V_{\mathfrak{z}}$,
it is simpler to consider $\partial_{i},\delta_{i}$ and $\Delta_{i}^{\pm}$
acting on $V_{\mathfrak{z}}$, and later restrict to $W_{\alpha}$.
The Weyl group of $G$ is $S_{3}$ (as permutation matrices), and
$G$ decomposes as
\[
G=BK=\coprod_{w\in A_{3}}BwE=BT\sqcup B\left(1\,2\right)T=\coprod_{w\in S_{3}}BwI.
\]
From $G=BK$ and \eqref{eq:UInd} we see that $\dim V_{\mathfrak{z}}^{K}\leq1$,
and in fact this is an equality since $\chi_{\mathfrak{z}}\big|_{B\cap K}\equiv1$,
hence $f^{K}\left(bk\right):=\widetilde{\chi}_{\mathfrak{z}}\left(b\right)$
is well defined. Similarly, $\dim V_{\mathfrak{z}}^{I}=6$, with basis
$\left\{ f_{w}^{I}\right\} _{w\in S_{3}}$ defined by $f_{w}^{I}\left(w'\right)=\delta_{w,w'}$,
and $\dim V_{\mathfrak{z}}^{E}=3$ with basis $\left\{ f_{w}^{E}\right\} _{w\in A_{3}}$,
where $f_{w}^{E}:=f_{w}^{I}+f_{w\cdot\left(1\,2\right)}^{I}$ satisfies
$f_{w}^{E}\left(w'\right)=\delta_{w,w'}$ for $w,w'\in A_{3}$. Finally,
$\dim V_{\mathfrak{z}}^{T}=2$ with basis $f_{w}^{T}\left(w'\right):=\delta_{w,w'}$
for $w,w'\in\left\{ \left(\,\right),\left(1\,2\right)\right\} $,
which satisfy
\begin{equation}
\begin{aligned}f_{\left(\,\right)}^{T} & =f_{\left(\,\right)}^{I}+\frac{1}{qz_{3}}f_{\left(3\,2\,1\right)}^{I}+\frac{z_{1}}{q}f_{\left(1\,2\,3\right)}^{I}\\
f_{\left(1\,2\right)}^{T} & =f_{\left(1\,2\right)}^{I}+z_{2}f_{\left(2\,3\right)}^{I}+\frac{1}{qz_{3}}f_{\left(1\,3\right)}^{I}\,;
\end{aligned}
\label{eq:T-basis}
\end{equation}
indeed, if $c_{w}$ is the coefficient of $f_{w}^{I}$ in $f_{\left(\,\right)}^{T}$,
then
\[
c_{\left(1\,2\,3\right)}=f_{\left(\,\right)}^{T}\left(\left(1\,2\,3\right)\right)=f_{\left(\,\right)}^{T}\left(\left(\begin{smallmatrix}\pi\\
 & 1\\
 &  & 1
\end{smallmatrix}\right)\sigma^{2}\right)=f_{\left(\,\right)}^{T}\left(\left(\begin{smallmatrix}\pi\\
 & 1\\
 &  & 1
\end{smallmatrix}\right)\right)=\widetilde{\chi}_{\mathfrak{z}}\left(\left(\begin{smallmatrix}\pi\\
 & 1\\
 &  & 1
\end{smallmatrix}\right)\right)=\frac{z_{1}}{q},
\]
and the other coefficients in \prettyref{eq:T-basis} are obtained
similarly.

Let $\Omega_{\mathfrak{z}}^{i}\left(\mathcal{B}\right)$ be the realization
of $V_{\mathfrak{z}}^{K_{i}}$ as a subspace of $\Omega^{i}\left(\mathcal{B}\right)$,
given by the explicit construction \prettyref{eq:UInd}. Any $f\in\Omega_{\mathfrak{z}}^{0}\left(\mathcal{B}\right)$
is determined by its value on $v_{0}$, namely $f=f\left(v_{0}\right)f^{K}$.
Similarly, the value on $e_{0}$, $e_{1}:=\left(3\,2\,1\right)e_{0}$
and $e_{2}:=\left(1\,2\,3\right)e_{0}$ determine a unique element
in $\Omega_{\mathfrak{z}}^{1}\left(\mathcal{B}\right)$, and likewise
for $t_{0}$, $t_{1}:=\left(1\,2\right)t_{0}$ and $\Omega_{\mathfrak{z}}^{2}\left(\mathcal{B}\right)$.
As $S_{3}\leq K$, one can compute the action of $\partial_{i}\big|_{\Omega_{\mathfrak{z}}^{i}\left(\mathcal{B}\right)}$
and $\delta_{i}\big|_{\Omega_{\mathfrak{z}}^{i-1}\left(\mathcal{B}\right)}$
by evaluation on $\mathrm{star}\left(v_{0}\right)$ alone (see Figure
\ref{fig:v0-star}).

\begin{figure}[h]
\begin{centering}
\includegraphics[scale=0.7]{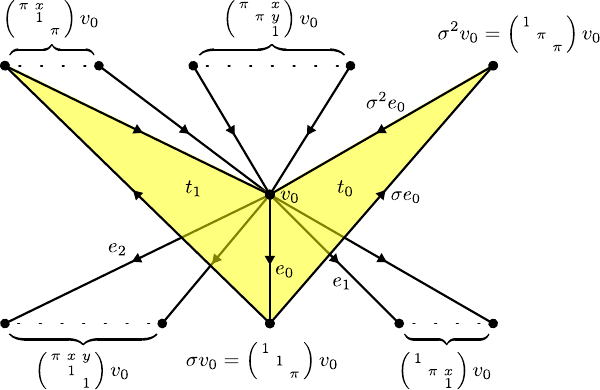}
\par\end{centering}
\caption{\label{fig:v0-star}The star of $v_{0}$ in $\mathcal{B}$.}
\end{figure}

For the basis $\mathfrak{B}^{E}=\left\{ f_{\left(\,\right)}^{E},f_{\left(3\,2\,1\right)}^{E},f_{\left(1\,2\,3\right)}^{E}\right\} $
one has 
\begin{equation}
\left[\delta_{1}\big|_{\Omega_{\mathfrak{z}}^{0}\left(\mathcal{B}\right)}\right]_{\mathfrak{B}^{E}}^{\negmedspace\left\{ \smash{f^{K}}\right\} }=\left(\begin{matrix}qz_{3}-1\\
z_{2}-1\\
\frac{z_{1}}{q}-1
\end{matrix}\right);\label{eq:coboundary_1_satake}
\end{equation}
for example, $\left(\delta_{1}f^{K}\right)\left(e_{1}\right)=f^{K}\left(\left(3\,2\,1\right)\sigma v_{0}\right)-f^{K}\left(v_{0}\right)=f^{K}\left(\left(\begin{smallmatrix}1\\
 & \pi\\
 &  & 1
\end{smallmatrix}\right)v_{0}\right)-1=z_{2}-1$, and $\left(\delta_{1}f^{K}\right)\left(e_{0}\right),\left(\delta_{1}f^{K}\right)\left(e_{2}\right)$
are computed similarly. We turn to $\partial_{1}$. The positive edges
with origin $v_{0}$ are $e_{0}$, $\left(\begin{smallmatrix}1\\
 & 1 & x\\
 &  & 1
\end{smallmatrix}\right)e_{1}$ for $x\in\mathbb{F}_{q}=\nicefrac{\mathcal{O}}{\pi\mathcal{O}}$,
and $\left(\begin{smallmatrix}1 & x & y\\
 & 1\\
 &  & 1
\end{smallmatrix}\right)e_{2}$ with $x,y\in\mathbb{F}_{q}$. As $\widetilde{\chi}_{\mathfrak{z}}$
is trivial on upper-triangular unipotent matrices, \prettyref{eq:boundary_1_1}
implies that for $f\in\Omega_{\mathfrak{z}}^{1}\left(\mathcal{B}\right)$
we have $\left(\mu\left(E\right)^{-\frac{1}{2}}\one_{K}f\right)\left(v_{0}\right)=f\left(e_{0}\right)+qf\left(e_{1}\right)+q^{2}f\left(e_{2}\right)$.
The positive edges entering $v_{0}$ are 
\[
\left[\sigma^{2}v_{0},v_{0}\right]=\sigma^{2}e_{0}=\left(\begin{smallmatrix} &  & 1\\
\pi\\
 & \pi
\end{smallmatrix}\right)e_{0}=\left(\begin{smallmatrix}1\\
 & \pi\\
 &  & \pi
\end{smallmatrix}\right)\left(1\,2\,3\right)e_{0}=\left(\begin{smallmatrix}1\\
 & \pi\\
 &  & \pi
\end{smallmatrix}\right)e_{2},
\]
and similarly $\left(\begin{smallmatrix}\pi & x\\
 & 1\\
 &  & \pi
\end{smallmatrix}\right)e_{1}$ and $\left(\begin{smallmatrix}\pi &  & x\\
 & \pi & y\\
 &  & 1
\end{smallmatrix}\right)e_{0}$ ($x,y\in\mathbb{F}_{q}$). By \prettyref{eq:boundary_1_2},
\begin{align*}
\left(\tfrac{1}{\sqrt{\mu\left(E\right)}}\one_{K\sigma^{2}E}f\right)\left(v_{0}\right) & =f\left(\left(\begin{smallmatrix}1\\
 & \pi\\
 &  & \pi
\end{smallmatrix}\right)e_{2}\right)+\sum_{x\in\mathbb{F}_{q}}f\left(\left(\begin{smallmatrix}\pi & x\\
 & 1\\
 &  & \pi
\end{smallmatrix}\right)e_{1}\right)+\sum_{x,y\in\mathbb{F}_{q}}f\left(\left(\begin{smallmatrix}\pi &  & x\\
 & \pi & y\\
 &  & 1
\end{smallmatrix}\right)e_{0}\right)\\
 & =z_{2}qz_{3}f\left(e_{2}\right)+q\cdot z_{1}z_{3}f\left(e_{1}\right)+q^{2}\cdot\frac{z_{1}}{q}z_{2}f\left(e_{0}\right),
\end{align*}
and in total (see \prettyref{eq:co_boundary_defs})
\begin{equation}
\left[\partial_{1}\big|_{\Omega_{\mathfrak{z}}^{1}\left(\mathcal{B}\right)}\right]_{\negmedspace\left\{ \smash{f^{K}}\right\} }^{\mathfrak{B}^{E}}=\left(\begin{matrix}\frac{q}{z_{3}}-1 & \frac{q}{z_{2}}-q & \frac{q}{z_{1}}-q^{2}\end{matrix}\right).\label{eq:boundary_1_satake}
\end{equation}
As $\Delta_{0}^{+}=\partial_{1}\delta_{1}$ and $\Delta_{1}^{-}=\delta_{1}\partial_{1}$,
we can now compute explicitly their action on the $\mathfrak{z}$-principal
series. Denoting $\widetilde{\mathfrak{z}}=\sum_{i=1}^{3}\left(z_{i}+z_{i}^{-1}\right)$,
we have by \prettyref{eq:coboundary_1_satake} and \prettyref{eq:boundary_1_satake}
\begin{align}
\Delta_{0}^{+}\big|_{\Omega_{\mathfrak{z}}^{0}\left(\mathcal{B}\right)} & =\left(\lambda^{K}\right):=\left(k_{0}-q\widetilde{\mathfrak{z}}\right),\label{eq:principle-series-e.v.-vertices}\\
\left[\Delta_{1}^{-}\big|_{\Omega_{\mathfrak{z}}^{1}\left(\mathcal{B}\right)}\right]_{\mathfrak{B}^{E}} & =\left(\begin{smallmatrix}{q}^{2}-qz_{{3}}-{\frac{q}{z_{{3}}}}+1 & -{q}^{2}z_{{3}}+{\frac{{q}^{2}z_{{3}}}{z_{{2}}}}+q-{\frac{q}{z_{{2}}}} & -{q}^{3}z_{{3}}+{\frac{{q}^{2}z_{{3}}}{z_{{1}}}}+{q}^{2}-{\frac{q}{z_{{1}}}}\\
\noalign{\smallskip}{\frac{qz_{{2}}}{z_{{3}}}}-z_{{2}}-{\frac{q}{z_{{3}}}}+1 & -qz_{{2}}+2\,q-{\frac{q}{z_{{2}}}} & -{q}^{2}z_{{2}}+{q}^{2}+{\frac{qz_{{2}}}{z_{{1}}}}-{\frac{q}{z_{{1}}}}\\
\noalign{\smallskip}-{\frac{q}{z_{{3}}}}-{\frac{z_{{1}}}{q}}+{\frac{z_{{1}}}{z_{{3}}}}+1 & q-z_{{1}}-{\frac{q}{z_{{2}}}}+{\frac{z_{{1}}}{z_{{2}}}} & {q}^{2}-qz_{{1}}-{\frac{q}{z_{{1}}}}+1
\end{smallmatrix}\right).\nonumber 
\end{align}
Observe that $\Delta_{0}^{+}$ agrees with the computation of the
spectrum of the Hecke operators in \cite{macdonald1979symmetric,Cartwright2002,li2004ramanujan,Lubotzky2005a},
as $\Delta_{0}^{+}=k_{0}\cdot I-\sum_{i=1}^{d-1}A_{i}$, where $A_{i}$
is the $i$-th Hecke operator on $\mathcal{B}_{d}$ (loc.\ cit.).
In effect, $\Delta_{1}^{-}$ can also be understood without the machinery
above, as it has the eigenvalue $\lambda^{K}$ corresponding to $\delta_{1}\Omega_{\mathfrak{z}}^{0}\left(\mathcal{B}\right)$
(since $\Delta_{1}^{-}\delta_{1}f^{K}=\delta_{1}\Delta_{0}^{+}f^{K}=\lambda^{K}\delta_{1}f^{K}$),
and two zeros (which come from $\partial_{2}\Omega_{\mathfrak{z}}^{2}\left(\mathcal{B}\right)$).
However, this machinery allows us to compute as easily the edge/triangle
spectrum: for $f\in\Omega_{\mathfrak{z}}^{1}\left(\mathcal{B}\right)$,
one has
\begin{align*}
\left(\delta_{2}f\right)\left(t_{0}\right) & =\sum\nolimits _{i=0}^{2}f\left(\sigma^{i}e_{0}\right)=f\left(e_{0}\right)+f\left(\left(\begin{smallmatrix}1\\
 & 1\\
 &  & \pi
\end{smallmatrix}\right)\left(3\,2\,1\right)e_{0}\right)+f\left(\left(\begin{smallmatrix}1\\
 & \pi\\
 &  & \pi
\end{smallmatrix}\right)\left(1\,2\,3\right)e_{0}\right)\\
 & =f\left(e_{0}\right)+qz_{3}f\left(e_{1}\right)+z_{2}qz_{3}f\left(e_{2}\right)\\
\left(\delta_{2}f\right)\left(t_{1}\right) & =\sum\nolimits _{i=0}^{2}f\left(\left(1\,2\right)\sigma^{i}e_{0}\right)=f\left(e_{0}\right)+f\left(\left(\begin{smallmatrix}1\\
 & 1\\
 &  & \pi
\end{smallmatrix}\right)\left(1\,2\,3\right)e_{0}\right)+f\left(\left(\begin{smallmatrix}\pi\\
 & 1\\
 &  & \pi
\end{smallmatrix}\right)\left(3\,2\,1\right)e_{0}\right)\\
 & =f\left(e_{0}\right)+qz_{3}f\left(e_{2}\right)+z_{1}z_{3}f\left(e_{1}\right),
\end{align*}
which gives $\left[\delta_{2}\big|_{\Omega_{\mathfrak{z}}^{1}\left(\mathcal{B}\right)}\right]_{\mathfrak{B}^{T}}^{\mathfrak{B}^{E}}=\left(\begin{matrix}1 & qz_{3} & qz_{2}z_{3}\\
1 & z_{1}z_{3} & qz_{3}
\end{matrix}\right)$, where $\mathfrak{B}^{T}$ is the ordered basis $f_{\left(\,\right)}^{T},f_{\left(1\,2\right)}^{T}$.
The triangles containing $e_{0}$ are obtained by adjoining $\sigma^{2}v_{0}$
(which gives $t_{0}$) and $\left(\begin{smallmatrix}\pi & x\\
 & 1\\
 &  & \pi
\end{smallmatrix}\right)v_{0}$ ($x\in\mathbb{F}_{q}$), giving $\left(\begin{smallmatrix}1 & x\\
 & 1\\
 &  & 1
\end{smallmatrix}\right)t_{1}$. This yields $\left(\partial_{2}f\right)\left(e_{0}\right)=f\left(t_{0}\right)+qf\left(t_{1}\right)$,
but for $e_{1},e_{2}$ we need to work a little harder, and use \prettyref{eq:T-basis}:
\begin{align*}
\left(\partial_{2}f\right)\left(e_{1}\right) & =\left(\partial_{2}f\right)\left(\left(3\,2\,1\right)e_{0}\right)=f\left(\left(3\,2\,1\right)t_{0}\right)+\sum_{x\in\mathbb{F}_{q}}f\left(\left(3\,2\,1\right)\left(\begin{smallmatrix}1 & x\\
 & 1\\
 &  & 1
\end{smallmatrix}\right)t_{1}\right)\\
 & =\frac{1}{qz_{3}}f\left(t_{0}\right)+f\left(\left(3\,2\,1\right)t_{1}\right)+\sum_{x\in\mathbb{F}_{q}^{\times}}f\left(\left(3\,2\,1\right)\left(\begin{smallmatrix}1 & x\\
 & 1\\
 &  & 1
\end{smallmatrix}\right)\left(1\,2\right)t_{0}\right)\\
 & =\frac{1}{qz_{3}}f\left(t_{0}\right)+f\left(\left(2\,3\right)t_{0}\right)+\sum_{x\in\mathbb{F}_{q}^{\times}}f\left(\left(\begin{smallmatrix}-\frac{1}{x} &  & 1\\
 & 1\\
 &  & x
\end{smallmatrix}\right)\left(3\,2\,1\right)\left(\begin{smallmatrix}1 & \frac{1}{x}\\
 & 1\\
 &  & 1
\end{smallmatrix}\right)t_{0}\right)\\
 & =\frac{1}{qz_{3}}f\left(t_{0}\right)+z_{2}f\left(t_{1}\right)+\sum_{x\in\mathbb{F}_{q}^{\times}}f\left(\left(3\,2\,1\right)t_{0}\right)=\frac{1}{z_{3}}f\left(t_{0}\right)+z_{2}f\left(t_{1}\right)\\
\left(\partial_{2}f\right)\left(e_{2}\right) & =\frac{z_{1}}{q}f\left(t_{0}\right)+\sum_{x\in\mathbb{F}_{q}}f\left(\left(\begin{smallmatrix}1\\
 & 1 & x\\
 &  & 1
\end{smallmatrix}\right)\left(1\,3\right)t_{0}\right)=\frac{z_{1}}{q}f\left(t_{0}\right)+\frac{1}{z_{3}}f\left(t_{1}\right),
\end{align*}
 so that $\left[\partial_{2}\big|_{\Omega_{\mathfrak{z}}^{2}\left(\mathcal{B}\right)}\right]_{\mathfrak{B}^{E}}^{\mathfrak{B}^{T}}=\left(\begin{matrix}1 & q\\
\nicefrac{1}{z_{3}} & z_{2}\\
\nicefrac{z_{1}}{q} & \nicefrac{1}{z_{3}}
\end{matrix}\right)$, giving 
\begin{align*}
\left[\Delta_{1}^{+}\big|_{\Omega_{\mathfrak{z}}^{1}\left(\mathcal{B}\right)}\right]_{\mathfrak{B}^{E}} & =\left(\begin{array}{ccc}
q+1 & \frac{q}{z_{2}}+qz_{{3}} & {q}^{2}z_{{3}}+\frac{q}{z_{1}}\\
\noalign{\smallskip}{z_{{2}}+\frac{1}{z_{{3}}}} & q+1 & \frac{q}{z_{1}}+qz_{{2}}\\
\noalign{\smallskip}{\frac{z_{1}}{q}+\frac{1}{z_{3}}} & \frac{1}{z_{2}}+z_{{1}} & q+1
\end{array}\right)\\
\left[\Delta_{2}^{-}\big|_{\Omega_{\mathfrak{z}}^{2}\left(\mathcal{B}\right)}\right]_{\mathfrak{B}^{T}} & =\left(\begin{array}{cc}
q+2 & \frac{q}{z_{1}}+qz_{{2}}+q\\
\noalign{\smallskip}\frac{1}{z_{2}}+z_{{1}}+1 & 2q+1
\end{array}\right).
\end{align*}
Recalling that $\lambda^{K}=k_{0}-q\widetilde{\mathfrak{z}}=2\left(q^{2}+q+1\right)-q\left(\sum z_{i}+z_{i}^{-1}\right)$,
\begin{equation}
\begin{aligned}\Spec\Delta_{1}^{+}\big|_{\Omega_{\mathfrak{z}}^{1}\left(\mathcal{B}\right)} & =\left\{ \lambda_{0}^{E},\lambda_{\pm}^{E}\right\} :=\left\{ 0,\tfrac{3}{2}\left(q+1\right)\pm\tfrac{1}{2}\sqrt{\left(q+1\right)^{2}+4q\left(2+\widetilde{\mathfrak{z}}\right)}\right\} \\
 & =\left\{ 0,\tfrac{3k_{1}}{2}\pm\sqrt{\left(\tfrac{3k_{1}}{2}\right)^{2}-\lambda^{K}}\right\} ,
\end{aligned}
\label{eq:principle-series-e.v.-edges}
\end{equation}
and again $\Spec\Delta_{2}^{-}\big|_{\Omega_{\mathfrak{z}}^{2}\left(\mathcal{B}\right)}=\left\{ \lambda_{\pm}^{E}\right\} $
as we have argued for $\Delta_{1}^{-}$. For $\Delta_{1}^{+}$, $\lambda_{0}^{E}=0$
is obtained on $\delta_{1}f^{K}$ (whose $f_{w}^{E}$ coefficients
were computed in \prettyref{eq:coboundary_1_satake}), and $\lambda_{\pm}^{E}$
are obtained on
\begin{align*}
f_{\pm}^{E} & =\left(\begin{array}{c}
2\left(z_{2}^{-2}+{\frac{z_{{1}}}{z_{{2}}}}\right){q}^{2}-2\left(z_{{3}}+1\right)q\\
\noalign{\smallskip}1-{q}^{2}z_{{1}}+q\left(z_{{1}}+\frac{2}{z_{2}}-\frac{2}{z_{3}}-1\right)\pm\left(qz_{{1}}-1\right)\sqrt{9k_{1}^{2}-4\lambda^{K}}\\
\noalign{\smallskip}qz_{1}\left(z_{2}^{-1}+2{z_{{1}}}+1\right)-{\frac{z_{{1}}}{z_{{2}}}}-z_{{1}}-\frac{2}{z_{2}}\pm\left(-z_{{1}}+{\frac{z_{{1}}}{z_{{2}}}}\right)\sqrt{9k_{1}^{2}-4\lambda^{K}}
\end{array}\right)^{T}\cdot\left(\begin{matrix}f_{\left(\,\right)}^{E}\\
\noalign{\medskip}f_{\left(3\,2\,1\right)}^{E}\\
\noalign{\medskip}f_{\left(1\,2\,3\right)}^{E}
\end{matrix}\right)\\
 & =2q\left(1+z_{2}+\frac{1}{z_{1}}\right)\partial_{2}f_{\left(\,\right)}^{T}+\left(q-1\pm\sqrt{9k_{1}^{2}-4\lambda^{K}}\right)\partial_{2}f_{\left(1\,2\right)}^{T}.
\end{align*}

\subsection{\label{sec:Unitary-Iwahori-spherical-repres}Unitary Iwahori-spherical
representations}

In general, an irreducible Iwahori-spherical representation is only
a subrepresentation of $V_{\mathfrak{z}}$. Denote by $W_{\mathfrak{z}}$
this subrepresentation (there is only one such). Tadic \cite{tadic1986classification}
classified the Satake parameters for which the representation $W_{\mathfrak{z}}$
admits a unitary structure, and in \cite{kang2010zeta} the possible
$\mathfrak{z}$ for $PGL_{3}\left(F\right)$ are listed, and a basis
for $W_{\mathfrak{z}}\leq V_{\mathfrak{z}}$ is computed explicitly,
using results from \cite{borel1976admissible,zelevinsky1980induced}.
It turns out that a unitary $E$-spherical $W_{\mathfrak{z}}$ is
of one of the following types:
\begin{enumerate}[label=(\alph{enumi})]
\item \label{enu:prin-unit}$\left|z_{i}\right|=1$ for $i=1,2,3$. In
this case $W_{\mathfrak{z}}=V_{\mathfrak{z}}$, and $\widetilde{\mathfrak{z}}\in\left[-3,6\right]$
gives $\lambda^{K}\in\left[k_{0}-6q,k_{0}+3q\right]$ and $\lambda_{\pm}^{E}\in\mathcal{I}_{\pm}$
(see \prettyref{eq:principle-series-e.v.-edges} and \prettyref{eq:good-bad-interval}).
\item \label{enu:bad-prin}$\mathfrak{z}=\left(c^{-2},cq^{a},cq^{-a}\right)$
for some $\left|c\right|=1$ and $0<a<\frac{1}{2}$. Here too $W_{\mathfrak{z}}=V_{\mathfrak{z}}$.
\item \label{enu:Li-e}$\mathfrak{z}=\left(\frac{c}{\sqrt{q}},c\sqrt{q},c^{-2}\right)$
for some $\left|c\right|=1$. In this case $W_{\mathfrak{z}}^{E}$
is one-dimensional, and spanned by $f_{-}^{E}$, which is proportional
to $qf_{\left(3\,2\,1\right)}^{E}-f_{\left(1\,2\,3\right)}^{E}$.
It corresponds to 
\begin{align*}
\lambda_{-}^{E} & =\frac{1}{2}\left(3k_{1}-\sqrt{k_{1}^{2}+8q+4q\left(\frac{c}{\sqrt{q}}+\overline{c}\sqrt{q}+c\sqrt{q}+\frac{\overline{c}}{q}+c^{-2}+c^{2}\right)}\right)\\
 & =\frac{1}{2}\left(3k_{1}-\sqrt{q^{2}+8q\sqrt{q}\Re\left(c\right)+2q+16q\Re\left(c\right)^{2}+1+8\sqrt{q}\Re\left(c\right)}\right)\\
 & =\frac{1}{2}\left(3k_{1}-\left(q+4\sqrt{q}\Re\left(c\right)+1\right)\right)=k_{1}-2\sqrt{q}\Re\left(c\right)
\end{align*}
which lies in $\mathcal{I}$ (see \prettyref{eq:excellent-interval}).
As $f_{-}^{E}$ is not $K$-fixed, $W_{\mathfrak{z}}^{K}=0$.
\item \label{enu:Li-d}$\mathfrak{z}=\left(c\sqrt{q},\frac{c}{\sqrt{q}},c^{-2}\right)$
for some $\left|c\right|=1$. Here $W_{\mathfrak{z}}^{E}=\left\langle f_{0}^{E},f_{+}^{E}\right\rangle $,
where $f_{+}^{E}$ is proportional to $\left(q+1\right)f_{\left(\,\right)}^{E}+\left(c^{2}+\frac{c}{\sqrt{q}}\right)\left(f_{\left(3\,2\,1\right)}^{E}+f_{\left(1\,2\,3\right)}^{E}\right)$,
and $\lambda_{+}^{E}=2k_{1}+2\sqrt{q}\Re\left(c\right)$ similarly
to the computation in type \prettyref{enu:Li-e}. This time $f^{K}$
is in $W_{\mathfrak{z}}^{E}$, and corresponds to $\lambda^{K}=k_{0}-2q\Re\left(\frac{\left(q+1\right)}{\sqrt{q}}c+c^{2}\right)$.
\item \label{enu:triv}$\mathfrak{z}=\left(q,1,{\scriptscriptstyle \frac{1}{q}}\right)$;
$W_{\mathfrak{z}}$ is the trivial representation $\rho:G\rightarrow\mathbb{C}^{\times}$,
and $W_{\mathfrak{z}}^{E}=W_{\mathfrak{z}}^{K}$ are spanned by $f^{K}=f_{+}^{E}$.
Since $f^{K}$ is constant and $f_{+}^{E}$ is a disorientation we
have $\lambda^{K}=0$ and $\lambda_{+}^{E}=3k_{1}$ (alternatively,
use \prettyref{eq:principle-series-e.v.-vertices} and \prettyref{eq:principle-series-e.v.-edges}).
\item \label{enu:color}$\mathfrak{z}=\left(\omega q,\omega,{\scriptscriptstyle \frac{\omega}{q}}\right)$
where $\omega=e^{\pm\frac{2\pi i}{3}}$; $W_{\mathfrak{z}}$ is the
one-dimensional representation $\rho\left(g\right)=\omega^{\tau\left(g\right)}$,
and $W_{\mathfrak{z}}^{K}=W_{\mathfrak{z}}^{E}=\left\langle f^{K}\right\rangle =\left\langle f_{0}^{E}\right\rangle $,
giving $\lambda^{K}=\frac{3k_{0}}{2}$.
\end{enumerate}
Apart from these there is the Steinberg $\mathrm{\left(Stn\right)}$
representation $\mathfrak{z}=\left({\scriptscriptstyle \frac{1}{q}},1,q\right)$.
It is not $E$-spherical, and $W^{T}$ is spanned by $f_{0}^{T}=qf_{\left(\,\right)}^{T}-f_{\left(1\,2\right)}^{T}$,
which is always in $\ker\partial_{2}=\ker\Delta_{2}^{-}$. (In \cite{kang2010zeta}
the twisted Steinberg representations $\mathfrak{z}=\left({\scriptscriptstyle \frac{\omega}{q}},\omega,\omega q\right)$
are also considered, but they do not contribute to $\Omega^{*}$ as
they have no $K$, $E$ or $T$-fixed vectors.)

Let $X=\Gamma\backslash\mathcal{B}$ be a non-tripartite Ramanujan
complex with $L^{2}\left(\Gamma\backslash G\right)\cong\bigoplus_{i}W_{\mathfrak{z_{i}}}$,
and denote by $N_{\left(t\right)}$ the number of $W_{\mathfrak{z}_{i}}$
of type $\left(t\right)$. These are computed in \cite{kang2010zeta}
for the tripartite case, and our arguments are similar. By the Ramanujan
assumption every Iwahori-spherical $W_{\mathfrak{z}_{i}}$ is either
tempered, which are the types \prettyref{enu:prin-unit}, \prettyref{enu:Li-e},
and $\mathrm{\left(Stn\right)}$, or finite-dimensional (types \prettyref{enu:triv},
\prettyref{enu:color}), so that $N_{\ref{enu:bad-prin}}=N_{\ref{enu:Li-d}}=0$.
The trivial representation \prettyref{enu:triv} always appears once
in $L^{2}\left(\Gamma\backslash G\right)$ as the constant functions,
so that $N_{\ref{enu:triv}}=1$.\footnote{This explains why $3k_{1}$ always appear in $\Spec\Delta_{1}^{+}$,
unlike the graph case, where $2k_{0}\in\Spec\Delta_{0}^{+}$ only
for bipartite quotients of $\mathcal{B}_{2}$.} Type \prettyref{enu:color} corresponds to $f\in L^{2}\left(\Gamma\backslash G\right)$
satisfying $f\left(\Gamma g\right)=\left(gf\right)\left(\Gamma\right)=\omega^{\tau\left(g\right)}f\left(\Gamma\right)$,
which is unique up to scaling, and well defined iff $\Gamma\leq\ker\tau$,
i.e.\ $X$ is tripartite. Therefore, $N_{\ref{enu:color}}=0$ and
\begin{align*}
n & =\dim\Omega^{0}\left(X\right)=\sum\nolimits _{i}\dim W_{\mathfrak{z}_{i}}^{K}=N_{\ref{enu:prin-unit}}+N_{\ref{enu:triv}}+N_{\ref{enu:color}}\\
\frac{nk_{0}}{2} & =\dim\Omega^{1}\left(X\right)=\sum\nolimits _{i}\dim W_{\mathfrak{z}_{i}}^{E}=3N_{\ref{enu:prin-unit}}+N_{\ref{enu:Li-e}}+N_{\ref{enu:triv}}+N_{\ref{enu:color}}
\end{align*}
together imply $N_{\ref{enu:prin-unit}}=n-1$ and $N_{\ref{enu:Li-e}}=n\left(q^{2}+q-2\right)+2$.
This is summarized in Table \ref{tab:rep-info}, together with the
tripartite case, and this also completes the proof of \prettyref{thm:ram-triangle-spec}.
For completeness, Table \ref{tab:rep-info} also shows $W^{T}$ for
each type. From $\frac{nk_{0}k_{1}}{6}=\dim\Omega^{2}\left(X\right)=2N_{\ref{enu:prin-unit}}+N_{\ref{enu:Li-e}}+N_{\ref{enu:triv}}+N_{\left(Stn\right)}$
one has $N_{\left(Stn\right)}=\sum_{i=-1}^{2}\left(-1\right)^{i}\left|X^{i}\right|=\widetilde{\chi}\left(X\right)$,
the reduced Euler characteristic of $X$.

\begin{table}[h]
\begin{centering}
\begin{tabular}{|c|c|c|c|c|c|c|c|c|}
\hline 
Type &  & $W^{K}$ & $\Delta_{0}^{+}$ e.v. & $W^{E}$ & $\Delta_{1}^{+}$ e.v. & $W^{T}$ & ${\mathrm{mult.}\atop \Gamma\leq\ker\tau}$ & ${\mathrm{mult.}\atop \Gamma\nleq\ker\tau}$\tabularnewline
\hline 
\hline 
\prettyref{enu:prin-unit} & {\small{}tempered} & $f^{K}$ & $k_{0}+q\widetilde{\mathfrak{z}}$ & $f_{0}^{E},f_{\pm}^{E}$ & $0,\tfrac{3k_{1}}{2}\!\pm\!\sqrt{\tfrac{9k_{1}^{2}}{4}-\lambda^{K}}$ & $\delta_{1}f_{\pm}^{E}$ & $n-3$ & $n-1$\tabularnewline
\hline 
\prettyref{enu:bad-prin} &  & $f^{K}$ & $k_{0}+q\widetilde{\mathfrak{z}}$ & $f_{0}^{E},f_{\pm}^{E}$ & $0,\tfrac{3k_{1}}{2}\!\pm\!\sqrt{\tfrac{9k_{1}^{2}}{4}-\lambda^{K}}$ & $\delta_{1}f_{\pm}^{E}$ & 0 & 0\tabularnewline
\hline 
\prettyref{enu:Li-e} & {\small{}tempered} & $0$ & - & $f_{-}^{E}$ & $k_{1}-2\sqrt{q}\Re\left(c\right)$ & $\delta_{1}f_{-}^{E}$ & ${nq^{2}+nq\atop -2n+6}$ & ${nq^{2}+nq\atop -2n+2}$\tabularnewline
\hline 
\prettyref{enu:Li-d} &  & $f^{K}$ & $k_{0}+q\widetilde{\mathfrak{z}}$ & $f_{0}^{E},f_{+}^{E}$ & $0,2k_{1}+2\sqrt{q}\Re\left(c\right)$ & $\delta_{1}f_{+}^{E}$ & 0 & 0\tabularnewline
\hline 
\prettyref{enu:triv} & {\small{}trivial} & $f^{K}$ & 0 & $f_{+}^{E}$ & $3k_{1}$ & $\delta_{1}f_{+}^{E}$ & 1 & 1\tabularnewline
\hline 
\prettyref{enu:color} & {\small{}fin.\ dim.} & $f^{K}$ & $\frac{3k_{0}}{2}$ & $f_{0}^{E}$ & 0 & 0 & 2 & 0\tabularnewline
\hline 
$\mathrm{\left(Stn\right)}$ & {\small{}tempered} & 0 & - & 0 & - & $f_{0}^{T}$ & $\widetilde{\chi}\left(X\right)$ & $\widetilde{\chi}\left(X\right)$\tabularnewline
\hline 
\end{tabular}
\par\end{centering}
\caption{\label{tab:rep-info}The representations appearing in $L^{2}\left(\Gamma\backslash G\right)$,
with the corresponding Laplacian eigenvalues, and the multiplicity
of appearance in the tripartite and non-tripartite Ramanujan cases.}
\end{table}

\section{\label{sec:Combinatorial-expansion}Combinatorial expansion}

\subsection{\label{sec:isop-expansion}Isoperimetric expansion}

The nontrivial spectrum of $\Delta_{0}^{+}$ on a non-tripartite Ramanujan
complex is highly concentrated, lying in a $k_{0}\pm O\left(\sqrt{k_{0}}\right)$
strip. The nontrivial $\Delta_{1}^{+}$-spectrum on $1$-cycles is
``almost concentrated'': there are $\approx nq^{2}$ eigenvalues
in a $k_{1}\pm O\left(\sqrt{k_{1}}\right)$ strip, but also $n-1$
eigenvalues at $2k_{1}\pm O\left(1\right)$ (and the trivial eigenvalue
$3k_{1}$). Nevertheless, having a concentrated vertex spectrum, and
edge spectrum bounded away from zero is enough to prove the Cheeger-type
inequality in Theorem \ref{thm:intro-cheeger-mixing}\eqref{enu:(Isoperimetry)-If-}:
For a partition of the vertices into sets $A_{0},A_{1},A_{2}$ of
sizes at least $\vartheta n,$ 
\[
\frac{\left|X\left(A_{0},A_{1},A_{2}\right)\right|n^{2}}{\left|A_{0}\right|\left|A_{1}\right|\left|A_{2}\right|}\geq2q^{3}-4q^{2.5}-C\cdot\frac{q^{2}}{\vartheta^{3}}.
\]

\begin{rems*}
\begin{enumerate}
\item This should be compared to the pseudo-random expectation: $X$ has
$\frac{1}{3!}$$nk_{0}k_{1}$ triangles, so its triangle density is
indeed $\frac{1}{3}n\left(q^{2}+q+1\right)\left(q+1\right)/{n \choose 3}\approx\frac{2q^{3}}{n^{2}}$. 
\item The restriction $\left|A_{i}\right|\geq\vartheta n$ is essential:
If $f\left(n\right)$ is any sub-linear function, one can take $A_{0}\subseteq X^{0}$
to be any set of size $f\left(n\right)$, $A_{1}$ some set containing
$\partial A_{0}=\left\{ v\,\middle|\,\mathrm{dist}\left(v,A_{0}\right)=1\right\} $,
and $A_{2}$ the rest of the vertices. Assuming $n$ is large enough
one has $\left|A_{0}\right|,\left|A_{1}\right|,\left|A_{2}\right|\geq f\left(n\right)$,
and $T\left(A_{0},A_{1},A_{2}\right)=\varnothing$.
\item Another Cheeger constant for complexes was suggested in \cite{parzanchevski2012isoperimetric},
and studied in \cite{Gundert2014}. However, it is trivial for clique
complexes, so we do not address it here.
\end{enumerate}
\end{rems*}

We prove Theorem \ref{thm:intro-cheeger-mixing}\eqref{enu:(Isoperimetry)-If-}
as part of Theorem \ref{thm:cheeger-gen-dim}, which applies to general
dimension.

\begin{proof}[Proof of Theorem \ref{thm:cheeger-gen-dim}]
For $f\in\Omega^{d-1}\left(X\right)$ defined by 
\[
f\left(\left[\sigma_{0}\:\sigma_{1}\:\ldots\:\sigma_{{d-1}}\right]\right)=\begin{cases}
\sgn\pi\left|A_{\pi\left({d}\right)}\right| & \exists\pi\in\mathrm{Sym}_{\left\{ 0\ldots{d}\right\} }\:\mathrm{with}\:\sigma_{i}\in A_{\pi\left(i\right)}\:\mathrm{for}\:0\leq i\leq{d-1}\\
0 & \mathrm{else,\:i.e.\:}\exists k,i\neq j\:\mathrm{with}\:\sigma_{i},\sigma_{j}\in A_{k},
\end{cases}
\]
it is shown in \cite[§4.1]{parzanchevski2012isoperimetric} that $\left\Vert \delta f\right\Vert ^{2}=\left|X\left(A_{0},\ldots,A_{{d}}\right)\right|n^{2}$.\footnote{While \cite{parzanchevski2012isoperimetric} assumes that $X$ has
a complete skeleton, this claim does not use this assumption.} For $f_{B}=\mathbb{P}_{B^{d-1}}f$ and $f_{Z}=\mathbb{P}_{Z_{d-1}}f$,
this gives
\begin{align*}
\left|X\left(A_{0},\ldots,A_{{d}}\right)\right|n^{2} & =\left\Vert \delta f\right\Vert ^{2}=\left\Vert \delta f_{Z}\right\Vert ^{2}=\left\langle \Delta_{d-1}^{+}f_{Z},f_{Z}\right\rangle \\
 & \geq\lambda_{d-1}\left\Vert f_{Z}\right\Vert ^{2}=\lambda_{d-1}\left(\left\Vert f\right\Vert ^{2}-\left\Vert f_{B}\right\Vert ^{2}\right).
\end{align*}
Denoting $\mathcal{K}=k_{0}\cdot\ldots\cdot k_{d-2}$ and $\mathcal{E}=\frac{\mu_{0}}{k_{0}}+\ldots+\frac{\mu_{d-2}}{k_{d-2}}$,
we have by \cite[Thm.\ 1.3]{Parzanchevski2013}
\begin{align*}
\left\Vert f\right\Vert ^{2} & =\sum_{i=0}^{d}\left|X\left(A_{0},\ldots,\widehat{A_{i}},\ldots,A_{d}\right)\right|\left|A_{i}\right|^{2}\geq\sum_{i=0}^{d}\left[\frac{\mathcal{K}}{n^{d-1}}\smash{\prod_{j\neq i}}\left|A_{j}\right|-c_{d-1}\mathcal{K}\mathcal{E}\max_{j\neq i}\left|A_{j}\right|\right]\left|A_{i}\right|^{2}\\
 & \geq\frac{\mathcal{K}}{n^{d-2}}\left(\prod_{i=0}^{d}\left|A_{i}\right|\right)-\left(d+1\right)c_{d-1}\mathcal{K}\mathcal{E}n^{3}\geq\mathcal{K}\left(n^{2-d}\prod_{i=0}^{d}\left|A_{i}\right|-\left(d+1\right)c_{d-1}\mathcal{E}n^{3}\right).
\end{align*}
Turning to $f_{B}$, let us denote $\mathfrak{D}=k_{0}\mathbb{P}_{B^{d-1}}-\Delta_{d-1}^{-}.$
Any linear maps $T:V\rightarrow W$ and $S:W\rightarrow V$ satisfy
$\left(\Spec TS\right)\backslash\left\{ 0\right\} =\left(\Spec ST\right)\backslash\left\{ 0\right\} $,
and thus 
\begin{align*}
\Spec\Delta_{d-1}^{-}\big|_{B^{d-1}} & =\Spec\Delta_{d-1}^{-}\backslash\left\{ 0\right\} =\Spec\Delta_{d-2}^{+}\backslash\left\{ 0\right\} =\Spec\Delta_{d-2}^{+}\big|_{B_{d-2}}\\
 & \subseteq\Spec\Delta_{d-2}^{+}\big|_{Z_{d-2}}\subseteq\left[k_{d-2}-\mu_{d-2},k_{d-2}+\mu_{d-2}\right].
\end{align*}
Together with $\Delta_{d-1}^{-}\big|_{Z_{d-1}}=0$ this implies $\left\Vert \mathfrak{D}\right\Vert \leq\mu_{d-2}$,
so that
\[
\left\Vert f_{B}\right\Vert ^{2}=\left\langle \mathbb{P}_{B^{d-1}}f,f\right\rangle \leq\frac{\left|\left\langle \mathfrak{D}f,f\right\rangle \right|+\left|\left\langle \Delta_{d-1}^{-}f,f\right\rangle \right|}{k_{d-2}}\leq\frac{\mu_{d-2}}{k_{d-2}}\left\Vert f\right\Vert ^{2}+\frac{1}{k_{d-2}}\left\Vert \partial f\right\Vert ^{2}.
\]

We note that $\partial f$ is supported on $\left(d-2\right)$-cells
with vertices in distinct blocks of the partition $\left\{ A_{i}\right\} $.
For a sequence of sets $B_{0},\ldots,B_{\ell}$, denote by $X^{j}\left(B_{0},\ldots,B_{\ell}\right)$
the set of \emph{$j$-galleries} in $B_{0},\ldots,B_{\ell}$, namely,
sequences of $j$-cells $\tau_{i}\in X\left(B_{i},\ldots,B_{i+j}\right)$
such that $\tau_{i}$ and $\tau_{i+1}$ intersect in a $\left(j-1\right)$-cell.
To shorten the formulae, we write $A_{\left[d\right]\backslash\{i,j\}}$
for $A_{0},\ldots,\widehat{A_{i}},\ldots,\widehat{A_{j}},\ldots,A_{d}.$We
have:
\begin{gather}
\begin{aligned}\left\Vert \partial f\right\Vert ^{2} & =\sum_{i<j}\:\sum_{\tau\in X\left(A_{\left[d\right]\backslash\{i,j\}}\right)}\negthickspace\negthickspace\negthickspace\negthickspace\left(\partial f\right)\left(\tau\right)^{2}=\sum_{i<j}\sum_{\tau}\left|\left|A_{j}\right|\smash{\sum_{\rho\in A_{i}}}\delta_{\tau\cup\rho\in X}-\left|A_{i}\right|\smash{\sum_{\eta\in A_{i}}}\delta_{\tau\cup\eta\in X}\vphantom{\bigg|}\right|^{2}\\
 & =\sum_{i<j}\left[{\left|A_{j}\right|^{2}\left|X^{d-1}\left(A_{i},A_{\left[d\right]\backslash\{i,j\}},A_{i}\right)\right|-2\left|A_{i}\right|\left|A_{j}\right|\left|X^{d-1}\left(A_{i},A_{\left[d\right]\backslash\{i,j\}},A_{j}\right)\right|\atop +\left|A_{i}\right|^{2}\left|X^{d-1}\left(A_{j},A_{\left[d\right]\backslash\{i,j\}},A_{j}\right)\right|}\right]
\end{aligned}
\label{eq:gallery-counting}
\end{gather}
Proposition 1.6 in \cite{Parzanchevski2013} estimates of the number
of $j$-galleries in $B_{0},\ldots,B_{\ell}$ when each $j+1$ tuple
$B_{i},B_{i+1},\ldots,B_{i+j+1}$ consists of disjoint sets, giving
\begin{equation}
\left|\left|A_{j}\right|^{2}\left|X^{d-1}\left(A_{i},A_{\left[d\right]\backslash\{i,j\}},A_{i}\right)\right|-\frac{\mathcal{K}k_{d-2}}{n^{d}}\left|A_{i}\right|\left|A_{j}\right|\prod_{k=0}^{d}\left|A_{k}\right|\right|\leq c_{d-2,d}\mathcal{K}k_{d-2}\mathcal{E}n^{3},\label{eq:this-eq}
\end{equation}
and similarly for the other summands in \prettyref{eq:gallery-counting}.
Observe that the middle term in \eqref{eq:this-eq} is the same in
all three cases, and thus cancels out in \prettyref{eq:gallery-counting},
giving  $\left\Vert \partial f\right\Vert ^{2}\leq4\binom{d+1}{2}c_{d-2,d}\mathcal{K}k_{d-2}\mathcal{E}n^{3}$.
In total,
\begin{multline*}
\frac{\left|X\left(A_{0},\ldots,A_{d}\right)\right|n^{d}}{\left|A_{0}\right|\ldots\left|A_{d}\right|}\geq\frac{\lambda_{d-1}n^{d-2}}{\left|A_{0}\right|\ldots\left|A_{d}\right|}\left(\left(1-\frac{\mu_{d-2}}{k_{d-2}}\right)\left\Vert f\right\Vert ^{2}-\frac{1}{k_{d-2}}\left\Vert \partial f\right\Vert ^{2}\right)\\
\geq\mathcal{K}\lambda_{d-1}\left(\left(1-\frac{\mu_{d-2}}{k_{d-2}}\right)\left(1-\left(d+1\right)c_{d-1}\mathcal{E}\frac{n^{d+1}}{\prod\left|A_{i}\right|}\right)-4\binom{d+1}{2}c_{d-2,d}\mathcal{E}\frac{n^{d+1}}{\prod\left|A_{i}\right|}\right)\\
\geq\mathcal{K}\lambda_{d-1}\left(1-\frac{\mu_{d-2}}{k_{d-2}}-\left(\left(d+1\right)c_{d-1}+4\binom{d+1}{2}c_{d-2,d}\right)\frac{\mathcal{E}n^{d+1}}{\prod_{i=0}^{d}\left|A_{i}\right|}\right)
\end{multline*}
and the theorem follows.

In the case of $d=2,$ one can work out the constants $c_{d-1}$ and
$c_{d-2,d}$ explicitly, and get the following inequality
\[
\frac{\left|T\left(A_{0},A_{1},A_{2}\right)\right|n^{2}}{\left|A_{0}\right|\left|A_{1}\right|\left|A_{2}\right|}\geq\lambda_{1}\left(k_{0}-\mu_{0}\left(1+\frac{10n^{3}}{9\left|A_{0}\right|\left|A_{1}\right|\left|A_{2}\right|}\right)\right).
\]

This implies Theorem \ref{thm:intro-cheeger-mixing}\eqref{enu:(Isoperimetry)-If-},
since by Theorem \ref{thm:Ram-spec} one has $k_{0}=2(q^{2}+q+1),\:\mu_{0}=6q$
and $\lambda_{1}=(q+1)-2\sqrt{q}$.
\end{proof}

\subsection{\label{sec:pseudo-randomness}Pseudo-randomness}

In this section we use not only the lower bound on the edge spectrum,
but the fact that it is concentrated in two narrow stripes, to show
a pseudo-random behavior of $2$-galleries.
\begin{thm}
\label{thm:mixing-general}Let $X$ be an $n$-vertex tripartite triangle
complex with vertex and edge degrees $k_{0}$ and $k_{1}$, such that
\begin{align*}
\Spec\Delta_{0}^{+}\big|_{Z_{0}} & \subseteq\left[k_{0}-\mu_{0},k_{0}+\mu_{0}\right]\cup\left\{ \tfrac{3k_{0}}{2}\right\} \qquad\mbox{and}\\
\Spec\Delta_{1}^{+}\big|_{Z_{1}} & \subseteq\left[k_{1}-\mu_{1},k_{1}+\mu_{1}\right]\cup\left[2k_{1}-\mu_{1},2k_{1}+\mu_{1}\right]\cup\left\{ 3k_{1}\right\} .
\end{align*}
If $A,B,C,D$ are disjoint sets of vertices of sizes $a,b,c,d$, respectively,
and each of $A\cup D$, $B$ and $C$ is contained in a different
block of the three-partition of $X$ (see Figure \ref{fig:2-gallery}),
then
\begin{equation}
\begin{aligned}\left|\left|X^{2}\left(A,B,C,D\right)\right|-\frac{27k_{0}k_{1}^{2}abcd}{2n^{3}}\right|\leq\frac{6\mu_{0}k_{1}^{2}\sqrt{abcd}}{k_{0}n}\left(\frac{3k_{0}\left(\sqrt{ab}+\sqrt{cd}\right)}{2n}+\mu_{0}\right)\\
+\left[\frac{2k_{1}^{2}\mu_{0}}{k_{0}}+\left(k_{1}+\mu_{1}\right)\mu_{1}\right]\sqrt[4]{abcd}\,\sqrt{\left(\tfrac{3k_{0}\sqrt{ab}}{2n}+\mu_{0}\right)\left(\tfrac{3k_{0}\sqrt{cd}}{2n}+\mu_{0}\right)}.
\end{aligned}
\label{eq:mixing}
\end{equation}
\end{thm}
\begin{figure}[h]
\captionsetup{format=plain, margin=1.5em, justification=raggedright}
\floatbox[{\capbeside}]{figure} {\caption{A $2$-gallery through $A$, $B$, $C$, $D$ in a tripartite triangle complex.}\label{fig:2-gallery}} {\includegraphics[scale=0.55]{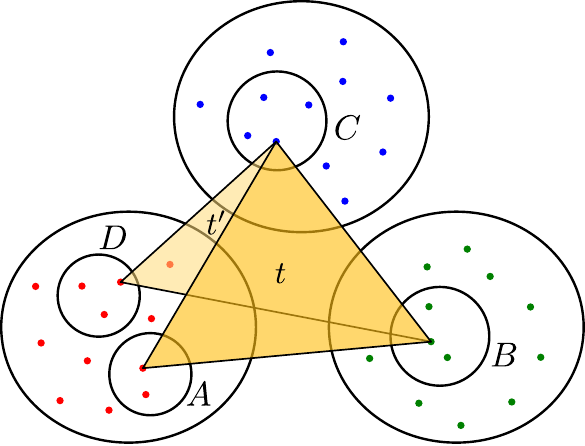}}
\end{figure}

It follows that if $a,b,c,d\leq\vartheta n$ (where $\vartheta\leq\frac{1}{3}$)
then the l.h.s.\ in \prettyref{eq:mixing} is bounded by 
\[
\left|\left|X^{2}\left(A,B,C,D\right)\right|-\frac{27k_{0}k_{1}^{2}abcd}{2n^{3}}\right|\leq\vartheta\left(9\vartheta+\frac{4\mu_{{0}}}{k_{{0}}}\right)\left(k_{{1}}\mu_{{0}}+k_{0}\mu_{{1}}\right)k_{1}n,
\]
and for Ramanujan complexes $k_{0}=2\left(q^{2}+q+1\right)$, $k_{1}=q+1$,
$\mu_{0}=6q$ and $\mu_{1}=2\sqrt{q}$, which gives Theorem \ref{thm:intro-cheeger-mixing}\eqref{enu:(Pseudo-randomness)-If-}.
The main term in \eqref{eq:mixing} agrees with the pseudo-random
expectation: given vertices $\alpha,\beta,\gamma,\delta$ in $A,B,C,D$,
respectively, the probability that $\beta\gamma$ is an edge in $X$
is $\frac{3k_{0}}{2n}$, and the $k_{1}$ triangles which contain
it have their third vertex in the block containing $A\cup D$. The
probability that $\alpha$ and $\delta$ are two of these is $\frac{k_{1}\left(k_{1}-1\right)}{\nicefrac{n}{3}\left(\nicefrac{n}{3}-1\right)}$,
so that $\mathbb{E}\left(\left|X^{2}\left(A,B,C,D\right)\right|\right)=\frac{3k_{0}}{2n}\cdot\frac{k_{1}\left(k_{1}-1\right)}{\nicefrac{n}{3}\left(\nicefrac{n}{3}-1\right)}\cdot abcd\approx\frac{27k_{0}k_{1}^{2}abcd}{2n^{3}}$.

\smallskip{}

We shall need a $c$-partite version of the expander mixing lemma,
where we say that a $k$-regular graph $\left(V,E\right)$ on $n$
vertices is $c$-partite if $V=V_{0}\sqcup\ldots\sqcup V_{c-1}$ with
$\left|V_{i}\right|=\frac{n}{c}$ so that $E\left(V_{i},V_{i}\right)=\varnothing$
and $\left|E\left(v,V_{j}\right)\right|=\frac{k}{c-1}$ for $v\in V_{i}$
and $j\neq i$. The functions $f_{j}\left(V_{\ell}\right)\equiv\smash{\exp\left(\frac{2\pi ij\ell}{c}\right)/\sqrt{n}}$
are orthonormal eigenfunctions of $\Delta_{0}^{+}$ with corresponding
eigenvalues 
\begin{equation}
\lambda_{j}=\begin{cases}
0 & j=0\\
\left(\frac{c}{c-1}\right)k & 0<j<c,
\end{cases}\label{eq:colored-spectrum}
\end{equation}
and we call $\left\{ \lambda_{0},\ldots\lambda_{c-1}\right\} $ the
\emph{partite spectrum}.
\begin{lem}
\label{lem:color-mixing}If the non-partite spectrum of a $c$-partite
$k$-regular graph on $n$ vertices is contained in $\left[k-\mu,k+\mu\right]$,
and $A\subseteq V_{i}$, $B\subseteq V_{j}$ for $i\neq j$, then
\begin{equation}
\left|\left|E\left(A,B\right)\right|-\frac{ck\left|A\right|\left|B\right|}{\left(c-1\right)n}\right|\leq\mu\sqrt{\left|A\right|\left|B\right|}.\label{eq:colored-mixing}
\end{equation}
\end{lem}
\begin{proof}
Assuming that $A\subseteq V_{0}$ and $B\subseteq V_{1}$, and denoting
by $\mathbb{P}_{W}$ the orthogonal projection on $W=\left\langle f_{0},\ldots,f_{c-1}\right\rangle ^{\bot}$,
we have 
\begin{align*}
\left|E\left(A,B\right)\right| & =\left\langle \left(kI-\Delta_{0}^{+}\right)\one_{A},\one_{B}\right\rangle =\sum_{j=0}^{c-1}\left(k-\lambda_{j}\right)\left\langle \one_{A},f_{j}\right\rangle \left\langle \one_{B},f_{j}\right\rangle +\left\langle \left(kI-\Delta_{0}^{+}\right)\mathbb{P}_{W}\one_{A},\one_{B}\right\rangle \\
 & =\frac{\left|A\right|\left|B\right|}{n}\sum_{j=0}^{c-1}\left(k-\lambda_{j}\right)\exp\left(\frac{2\pi ij}{c}\right)+\left\langle \left(kI-\Delta_{0}^{+}\right)\mathbb{P}_{W}\one_{A},\one_{B}\right\rangle ,
\end{align*}
and \prettyref{eq:colored-mixing} follows by \prettyref{eq:colored-spectrum}
and $\left\Vert \left(kI-\Delta_{0}^{+}\right)\big|_{W}\right\Vert \leq\mu$.
\end{proof}

\begin{proof}[Proof of \prettyref{thm:mixing-general}]
Denote by $U^{+}$ the span of the $\Delta_{1}^{+}$-eigenforms with
eigenvalues in $\left[k_{1}-\mu_{1},k_{1}+\mu_{1}\right]\cup\left[2k_{1}-\mu_{1},2k_{1}+\mu_{1}\right]$,
and by $\eta$ a normalized $3k_{1}$-eigenform for $\Delta_{1}^{+}$,
so that $\Omega^{1}\left(X\right)=B^{1}\oplus U^{+}\oplus\left\langle \eta\right\rangle $.
Denoting $p\left(x\right)=\left(x-k_{1}\right)\left(x-2k_{1}\right)$,
$p\left(\Delta_{1}^{+}\right)$ acts on $B^{1}\oplus\left\langle \eta\right\rangle $
as the scalar $2k_{1}^{2}$, and 
\[
\left\Vert p\left(\Delta_{1}^{+}\right)\big|_{U^{+}}\right\Vert \leq\max\left\{ \left|p\left(\lambda\right)\right|\,\middle|\,\lambda\in\begin{smallmatrix}\left[k_{1}-\mu_{1},k_{1}+\mu_{1}\right]\,\cup\,\\
\left[2k_{1}-\mu_{1},2k_{1}+\mu_{1}\right]
\end{smallmatrix}\right\} =\left(k_{1}+\mu_{1}\right)\mu_{1}.
\]
Say that two directed edges are \emph{neighbors }if they have a common
origin or a common terminus, and their union (as a cell) is in $X^{2}$.
We denote this by $e\sim e'$, and define $\mathcal{A}:\Omega^{1}\left(X\right)\rightarrow\Omega^{1}\left(X\right)$
by $\left(\mathcal{A}f\right)\left(e\right)=\sum_{e'\sim e}f\left(e'\right)$.
The upper Laplacian satisfies $\Delta_{1}^{+}=k_{1}\cdot I-\mathcal{A}$
(see \cite{Parzanchevski2013}), and it follows that $p\left(\Delta_{1}^{+}\right)=\mathcal{A}^{2}+k_{1}\mathcal{A}$.
Define $\one_{AB}\in\Omega^{1}\left(X\right)$ by 
\[
\one_{AB}\left(vw\right)=\begin{cases}
1 & v\in A,w\in B\\
-1 & v\in B,w\in A\\
0 & otherwise
\end{cases}
\]
and similarly $\one_{CD}$. We claim that 
\begin{equation}
\left|X^{2}\left(A,B,C,D\right)\right|=\left\langle p\left(\Delta_{1}^{+}\right)\one_{AB},\one_{CD}\right\rangle .\label{eq:two-galleries}
\end{equation}
Indeed, edges in $E\left(A,B\right)$ have no neighbors in $E\left(C,D\right)$,
so that $\left\langle \mathcal{A}\one_{AB},\one_{CD}\right\rangle =0$,
and $\mathcal{A}$ is self-adjoint (since $\Delta_{1}^{+}$ is), giving
\begin{gather}
\left\langle p\left(\Delta_{1}^{+}\right)\one_{AB},\one_{CD}\right\rangle =\left\langle \mathcal{A}\one_{AB},\mathcal{A}\one_{CD}\right\rangle =\sum_{{e,e',e''\in X^{1}\atop e'\sim e\sim e''}}\one_{AB}\left(e'\right)\one_{CD}\left(e''\right).\label{eq:sum_1AB_1CD}
\end{gather}
The nonzero terms in this sum come from edges $e$ which have neighbors
$e'\in E\left(A,B\right)$ and $e''\in E\left(C,D\right)$, and it
follows that $e\in E\left(B,C\right)$. Thus, $\left(e'\cup e,e\cup e''\right)$
is a $2$-gallery in $X^{2}\left(A,B,C,D\right)$, and observing that
$\one_{AB}\left(e'\right)=\one_{CD}\left(e''\right)\left(=\pm1\right)$
it contributes one to \prettyref{eq:sum_1AB_1CD}. On the other hand,
for every gallery $\left(t,t'\right)\in X^{2}\left(A,B,C,D\right)$,
the edges $e'=t\backslash C$, $e=t\cap t'$, $e''=t\backslash B$
form such a triplet, and we obtain \prettyref{eq:two-galleries}.
On the spectral side, 
\[
\left\langle p\left(\Delta_{1}^{+}\right)\one_{AB},\one_{CD}\right\rangle =2k_{1}^{2}\left\langle \mathbb{P}_{B^{1}\oplus\left\langle \eta\right\rangle }\one_{AB},\one_{CD}\right\rangle +\left\langle p\left(\Delta_{+}^{1}\right)\mathbb{P}_{U^{+}}\one_{AB},\one_{CD}\right\rangle ,
\]
and the last term is bounded by 
\begin{equation}
\left\Vert p\left(\Delta_{+}^{1}\right)\big|_{U^{+}}\right\Vert \left\Vert \one_{AB}\right\Vert \left\Vert \one_{CD}\right\Vert \leq\left(k_{1}+\mu_{1}\right)\mu_{1}\sqrt{E_{AB}E_{CD}},\label{eq:mu1-error}
\end{equation}
where $E_{ST}:=\left|E\left(S,T\right)\right|$. As $\eta$ has constant
sign on $V_{0}\rightarrow V_{1}\rightarrow V_{2}\rightarrow V_{0}$,
\begin{equation}
2k_{1}^{2}\left\langle \mathbb{P}_{\left\langle \eta\right\rangle }\one_{AB},\one_{CD}\right\rangle =2k_{1}^{2}\left\langle \one_{AB},\eta\right\rangle \left\langle \eta,\one_{CD}\right\rangle =\frac{4k_{1}^{2}E_{AB}E_{CD}}{k_{0}n},\label{eq:disor-contrib}
\end{equation}
and we are left to analyze $\mathbb{P}_{B^{1}}\one_{AB}$. As in the
non-tripartite case, one has $\Spec\Delta_{-}^{1}\big|_{B^{1}}=\Spec\Delta_{+}^{0}\big|_{B_{0}}$,
but now the latter comprises not only eigenvalues in $\left[k_{0}-\mu_{0},k_{0}+\mu_{0}\right]$,
but also $\frac{3k_{0}}{2}$ (twice, see \prettyref{thm:ram-triangle-spec}).
If $\omega=\exp\left(\tfrac{2\pi i}{3}\right)$ and 
\[
\xi\left(vw\right)=\left\{ \begin{aligned}\sqrt{\nicefrac{2}{k_{0}n}} &  & v\in V_{0},w\in V_{1} &  &  &  &  & -\sqrt{\nicefrac{2}{k_{0}n}} &  & w\in V_{0},v\in V_{1}\\
\omega\sqrt{\nicefrac{2}{k_{0}n}} &  & v\in V_{1},w\in V_{2} &  &  &  &  & -\omega\sqrt{\nicefrac{2}{k_{0}n}} &  & w\in V_{1},v\in V_{2}\\
\overline{\omega}\sqrt{\nicefrac{2}{k_{0}n}} &  & v\in V_{2},w\in V_{0} &  &  &  &  & -\overline{\omega}\sqrt{\nicefrac{2}{k_{0}n}} &  & w\in V_{2},v\in V_{0},
\end{aligned}
\right.
\]
then $\left\{ \xi,\overline{\xi}\right\} $ is an orthonormal basis
for the $\frac{3k_{0}}{2}$-eigenspace of $\Delta_{1}^{-}$. Denote
by $U^{-}$ the space spanned by the $\Delta_{1}^{-}$-eigenforms
with eigenvalue in $\left[k_{0}-\mu_{0},k_{0}+\mu_{0}\right]$. By
the action of each summand in 
\[
\mathfrak{D}':=k_{0}\mathbb{P}_{B^{1}}+\frac{k_{0}}{2}\mathbb{P}_{\left\langle \xi,\overline{\xi}\right\rangle }-\Delta_{1}^{-}
\]
on each of the terms in the orthogonal decomposition $\Omega^{1}\left(X\right)=Z_{1}\oplus U^{-}\oplus\left\langle \xi,\overline{\xi}\right\rangle $
we see that $\left\Vert \mathfrak{D}'\right\Vert \leq\mu_{0}$. Due
to the fact that $\partial_{1}\one_{AB}$ and $\partial_{1}\one_{CD}$
are supported on different vertices, $\left\langle \Delta_{1}^{-}\one_{AB},\one_{CD}\right\rangle $
vanishes, and together with
\[
\left\langle \mathbb{P}_{\left\langle \xi,\overline{\xi}\right\rangle }\one_{AB},\one_{CD}\right\rangle =2\Re\left(\left\langle \one_{AB},\xi\right\rangle \left\langle \xi,\one_{CD}\right\rangle \right)=-\frac{2E_{AB}E_{CD}}{k_{0}n}
\]
(and $\mathbb{P}_{B^{1}}=\frac{1}{k_{0}}\Delta_{1}^{-}-\frac{1}{2}\mathbb{P}_{\left\langle \xi,\overline{\xi}\right\rangle }+\frac{1}{k_{0}}\mathfrak{D}'$)
this gives 
\[
\left|2k_{1}^{2}\left\langle \mathbb{P}_{B^{1}}\one_{AB},\one_{CD}\right\rangle -\frac{2k_{1}^{2}E_{AB}E_{CD}}{k_{0}n}\right|\leq\frac{2k_{1}^{2}}{k_{0}}\left\langle \mathfrak{D}'\one_{AB},\one_{CD}\right\rangle \leq\frac{2\mu_{0}k_{1}^{2}\sqrt{E_{AB}E_{CD}}}{k_{0}}.
\]
Combining this with \prettyref{eq:mu1-error} and \prettyref{eq:disor-contrib}
we conclude that 
\[
\left|\left|X^{2}\left(A,B,C,D\right)\right|-\frac{6k_{1}^{2}E_{AB}E_{CD}}{k_{0}n}\right|\leq\left(\frac{2\mu_{0}k_{1}^{2}}{k_{0}}+\left(k_{1}+\mu_{1}\right)\mu_{1}\right)\sqrt{E_{AB}E_{CD}}.
\]
This estimates $\left|X^{2}\left(A,B,C,D\right)\right|$ in terms
of $E_{AB}$ and $E_{CD}$. To have an estimate in terms of $a,b,c$
and $d$ we use \prettyref{lem:color-mixing}, which gives $\left|E_{AB}-\frac{3k_{0}ab}{2n}\right|\leq\mu_{0}\sqrt{ab}$
and similarly for $E_{CD}$, and the theorem follows.
\end{proof}
Various applications of a triangle mixing lemma can be adjusted to
use our gallery mixing lemma. We demonstrate this below for chromatic
numbers, and other examples are Gromov's overlap property, along the
lines of \cite{FGL+11,Parzanchevski2013}, and the crossing numbers
of complexes, as discussed in \cite[§8.1]{Gundert2013}. Nevertheless,
the question of triangle pseudorandomness remains interesting, and
should give better results if it does hold. The fact that most of
the spectrum is concentrated in the strip $\mathcal{I}$ (see \prettyref{thm:ram-triangle-spec})
gives hope that this can be done by analyzing the combinatorics of
eigenforms which occur in the principal series (type \prettyref{enu:prin-unit}
in \prettyref{sec:Unitary-Iwahori-spherical-repres}), and showing
that their contribution is negligible. 

\subsection{Chromatic Number}

As an application of Theorem \ref{thm:mixing-general} we prove Theorem
\ref{thm:intro-cheeger-mixing}\eqref{enu:chrom}, which bounds the
chromatic number of non-tripartite Ramanujan complexes.
\begin{proof}[Proof of Theorem \ref{thm:intro-cheeger-mixing}(3)]
Write $X=\Gamma\backslash\mathcal{B}_{3}$ and let $\widehat{\Gamma}=\Gamma\cap\ker\tau$.
This is a normal subgroup of $\Gamma$ of index three, and $\widehat{X}:=\widehat{\Gamma}\backslash\mathcal{B}_{3}\overset{\pi}{\twoheadrightarrow}X$
is a tripartite three-cover. If the chromatic number of $X$ is $\chi$,
we can find a set $N\subseteq X^{0}$ of size $\frac{n}{\chi}$ with
$T\left(N,N,N\right)=\varnothing$. Let $N_{i}=\left\{ v\in\widehat{X}^{0}\,\middle|\,\pi\left(v\right)\in N,\tau\left(v\right)=i\right\} $,
and take $A=N_{0}$, $B=N_{1}$, $C=N_{2}$ and $D\subseteq\left\{ v\in\widehat{X}^{0}\,\middle|\,\tau\left(v\right)=0\right\} \backslash N_{0}$
such that $|D|=\frac{n}{2}$. Since the set $T\left(N,N,N\right)$
is empty, $X^{2}\left(A,B,C,D\right)$ is empty as well. Therefore
the l.h.s.\ of \prettyref{eq:mixing} reads $\frac{q^{4}n}{2\chi^{3}}=\frac{q^{4}abcd}{n^{3}}\leq\frac{27k_{0}k_{1}^{2}abcd}{2\left(3n\right)^{3}}$.
Assume to the contrary that $\chi<\frac{\sqrt[3]{q}}{5}$. Then the
r.h.s.\ of \prettyref{eq:mixing} is bounded by
\[
\frac{nq^{3.5}}{\chi^{1.5}\sqrt{2}}\left(2+\frac{1}{125}\left(264+255\sqrt{2}\right)\right)<\frac{nq^{3.5}}{\chi^{1.5}}\cdot\frac{7}{\sqrt{2}},
\]
so that \prettyref{thm:mixing-general} implies $\sqrt{q}<7\sqrt{2}\chi^{1.5}$,
which contradicts the assumption.
\end{proof}

\bibliographystyle{amsalpha}
\bibliography{mybib}

\bigskip{}

\parbox[t]{0.55\columnwidth}{%
\noun{Einstein Institute of Mathematics}

\noun{The Hebrew University of Jerusalem }

\noun{Jerusalem, 91904, Israel}

E-mail: \texttt{kost.golubev@mail.huji.ac.il}%
}%
\parbox[t]{0.44\columnwidth}{%
\noun{School of Mathematics}

\noun{Institute for Advanced Study}

\noun{Princeton, NJ 08540, USA}

E-mail: \texttt{parzan@ias.edu}%
}
\end{document}